\documentclass[twoside,11pt]{article}

\usepackage{blindtext}
\usepackage{graphicx} 
\usepackage{geometry}
\usepackage{amsfonts, amsmath, amssymb, amsthm}
\usepackage{bm}
\usepackage{url}
\usepackage{subcaption}
\usepackage{enumitem}
\usepackage{mathrsfs}

\newtheorem{assumption}{Assumption}
\newtheorem{claim}{Claim}

\newcommand{\gk}{\nabla ^\top g(f(w_k))}

\newcommand{\argmin}[1]{\underset{#1}{\arg\min}}
\newcommand{\grad}{\rr{grad}}
\newcommand{\Exp}{\rr{Exp}}

\newcommand{\bE}{\mathbb{E}}

\newcommand{\bN}{\mathbb{N}}

\newcommand{\bR}{\mathbb{R}}
\newcommand{\bS}{\mathbb{S}}

\newcommand{\cB}{\mathcal{B}}
\newcommand{\cC}{\mathcal{C}}
\newcommand{\cD}{\mathcal{D}}

\newcommand{\cM}{\mathcal{M}}
\newcommand{\cN}{\mathcal{N}}
\newcommand{\cO}{\mathcal{O}}

\newcommand{\cR}{\mathcal{R}}
\newcommand{\cS}{\mathcal{S}}
\newcommand{\cT}{\mathcal{T}}

\newcommand{\mdot}{\mathord{\cdot}}

\newcommand{\dt}[1]{\frac{\mathrm{d}{#1}}{\mathrm{d}t}}
\newcommand{\rr}[1]{\mathrm{#1}}
\newcommand{\Ht}{\bm{H}_\rho(\alpha f (w_t))}
\newcommand{\rank}{\mathrm{rank}}
\usepackage[preprint]{jmlr2e}

\usepackage{lastpage}
\jmlrheading{X}{2025}{1-\pageref{LastPage}}{8/25; Revised X/25}{X/25}{21-0000}{Author One and Author Two}

\ShortHeadings{Riemannian Analysis of Gauss-Newton for Neural Networks}{Cayci}
\firstpageno{1}

\begin{document}

\title{A Riemannian Optimization Perspective of the Gauss-Newton Method for Feedforward Neural Networks}

\author{\name Semih Cayci \email {cayci@mathc.rwth-aachen.de} \\
       \addr Department of Mathematics\\
       RWTH Aachen University\\
       Aachen 52062, Germany
       }

\editor{}

\maketitle

\begin{abstract}
In this work, we establish non-asymptotic convergence bounds for the Gauss-Newton method in training neural networks with smooth activations. In the underparameterized regime, the Gauss-Newton gradient flow in parameter space induces a Riemannian gradient flow on a low-dimensional embedded submanifold of the function space. Using tools from Riemannian optimization, we establish geodesic Polyak-\L{}ojasiewicz and Lipschitz-smoothness conditions for the loss under appropriately chosen output scaling, yielding geometric convergence to the optimal in-class predictor at an explicit rate independent of the conditioning of the Gram matrix. In the overparameterized regime, we propose adaptive, curvature-aware regularization schedules that ensure fast geometric convergence to a global optimum at a rate independent of the minimum eigenvalue of the neural tangent kernel and, locally, of the modulus of strong convexity of the loss. These results demonstrate that Gauss–Newton achieves accelerated convergence rates in settings where first-order methods exhibit slow convergence due to ill-conditioned kernel matrices and loss landscapes.
\end{abstract}


\section{Introduction}
First-order optimization methods typically suffer from slow convergence in deep learning due to factors including ill-conditioned data geometry and loss landscapes \citep{shalev2017failures}. In order to address these drawbacks, geometry-aware preconditioned optimization methods have attracted significant attention in deep learning as a way to accelerate and stabilize training. Particularly, the Gauss-Newton (GN) method has been a focal point of practical and theoretical interest due to its strong empirical performance in large-scale deep learning problems \citep{abreu2025potential, liu2025adam, martens2020new, ren2019efficient, botev2017practical}. Remarkably, recent empirical studies report up to 5.4x reduction in training iterations under GN on large language models compared to optimizers such as AdamW, SOAP and Muon \citep{abreu2025potential}. Despite its impressive empirical success, a concrete theoretical understanding of the Gauss-Newton method in deep learning is still in a nascent stage. Particularly, the convergence and optimality of this method and the benefits of preconditioning in the over- and underparameterized learning settings remain largely unexplored.

\subsection{Main Contributions}
In this work, we investigate the convergence behavior of the Gauss-Newton method and the impact of Gauss-Newton preconditioning in deep learning from a novel Riemannian optimization perspective. Our results highlight that Gauss-Newton preconditioning effectively mitigates slow convergence of the first-order methods due to ill-conditioned kernel matrices and loss landscapes under appropriate scaling and damping choices.

\textbullet\quad\textbf{Convergence of Gauss–Newton in the overparameterized regime.} In this regime, we prove the fast global convergence of the Gauss–Newton method with regularization. The main results in this section are Theorems \ref{thm:lm-convergence} and \ref{thm:non-quad-out}, and Corollary \ref{cor:lm-convergence-ada}. In Section \ref{sec:op-deep}, we present the results for deep neural networks.

     \begin{itemize}
         \item[--] \textit{Adaptive regularization schedules for Gauss-Newton.} We propose data-dependent and curvature-aware adaptive regularization schedules, which hybridize Gauss–Newton and gradient descent to mitigate ill-conditioned data distributions and loss landscapes. Under these schedules, we establish fast convergence to global optima (Theorems \ref{thm:lm-convergence}, \ref{thm:non-quad-out} and Corollary \ref{cor:lm-convergence-ada} for continuous-time, and Theorem \ref{thm:lm-convergence-dt} for discrete-time analysis). We characterize the implicit bias of the regularized Gauss-Newton method in the neural tangent kernel regime in Proposition \ref{prop:implicit-bias} for the quadratic loss function.

         Our analysis reveals that the Gauss-Newton method with our proposed regularization schedules achieves rapid geometric convergence rates that are notably independent of the smallest eigenvalue of the kernel matrix. This results in a substantial improvement in convergence rates, particularly when dealing with large datasets with small data separation. Furthermore, we show that switching to a curvature- and data-dependent regularization schedule upon entering a neighborhood of the optimum that we explicitly characterize, \textsc{GN} achieves fast convergence independent of the modulus of strong convexity of the loss function.

     \end{itemize}

 \textbullet\quad\textbf{Convergence of Gauss–Newton in the underparameterized regime.} We leverage the rich theory of Riemannian optimization for a fine-grained geometric convergence and optimality analysis of (unregularized) GN in this regime. To the best of our knowledge, our work establishes the first non-asymptotic performance guarantees for Gauss-Newton in this regime. The main results in this setting are stated in Theorem \ref{thm:geodesic-regularity} and Theorem \ref{thm:gf-convergence}.
    \begin{itemize}
        \item[--] \textit{Riemannian gradient flow in the function space.}  We show that the Gauss–Newton gradient flow induces a \emph{Riemannian} gradient flow on a low-dimensional smooth embedded submanifold $\cM$ of the predictor space $\bR^n$ (Prop. \ref{prop:submanifold} and Prop. \ref{prop:riemannian-gf}).
        \item[--] \textit{Geodesic convexity and smoothness of the loss function.} We establish \emph{geodesic} strong convexity and Lipschitz smoothness of the loss function on a level set $\mathcal{S}\subset\mathcal{M}$ that contains the optimization trajectory (Theorem \ref{thm:geodesic-regularity}). Our variational analysis demonstrates how the output scaling explicitly controls the curvature, thereby yielding the key geodesic regularity properties required for fast geometric convergence.
        \item[--] \textit{Convergence and in-class optimality for Gauss–Newton.} Leveraging the geodesic regularity conditions, we prove that the metric is non-degenerate indefinitely, and thus the Gauss–Newton method yields convergence of the \textbf{last-iterate} to the optimal in-class predictor at a \textbf{geometric rate} independent of the conditioning of the Gram matrix \textbf{without} any explicit regularization (Theorem \ref{thm:gf-convergence}).
        \item[--] \textit{Inductive bias and curvature.} The initialization and the output scaling factor play a vital role in the convergence rate, curvature and the inductive bias. We explicitly characterize these impacts in Theorem \ref{thm:geodesic-regularity} and Remarks \ref{remark:curvature} and \ref{remark:alpha}).
    \end{itemize}

\subsection{Related Works}

\paragraph{Analysis of the Gauss-Newton method.} The Gauss-Newton method has a long history in numerical linear analysis \citep{saad2003iterative} and nonlinear least-squares \citep{nocedal1999numerical}. It has recently attracted renewed attention and become a focal point of practical and theoretical interest because of its success in deep learning \citep{korbit2024exact, tan2019review, botev2017practical} and scientific machine learning \citep{rathore2024challenges, hao2024gauss, muller2023achieving}. Despite its empirical success, its theoretical underpinnings in deep learning have been largely unknown. Its convergence has been investigated in a number of recent works \citep{cai2019gram, zhang2019fast, arbel2024rethinking, adeoye2024regularized, zhao2024theoretical, jia2024globally}, which consider the Gauss-Newton method in the overparameterized setting. In this work, we provide finite-time and finite-width analyses for both the underparameterized and overparameterized regimes, and identify provably good choices of regularization and output scaling parameters. In the overparameterized regime, we study various adaptive curvature-aware damping (regularization) schemes and theoretically prove the benefits of the regularized Gauss-Newton method in training neural networks in the lazy training regime. In the underparameterized setting, we develop a Riemannian optimization analysis for the Gauss-Newton method, which is fundamentally different from the existing works. 

\paragraph{Optimization in the lazy training regime.} The original works in the lazy training regime analyze the convergence of gradient descent for overparameterized neural networks \citep{du2018gradient, jacot2018neural, chizat2019lazy}. Our analysis builds on the analysis proposed in \cite{chizat2019lazy, du2018gradient}, and extends it to analyze the Gauss-Newton method for both over- and underparameterized neural networks. In the underparameterized regime, deviating significantly from the existing works, we integrate tools from the Riemannian optimization theory to analyze the Gauss-Newton dynamics in training neural networks. In a number of works \citep{ji2020polylogarithmic, cayci2024convergence, cai2019neural}, convergence of first-order methods in the underparameterized regime was investigated in the near-initialization regime. These results establish near-optimality results for first-order methods (i) under explicit regularization in the form of projection or early stopping, (ii) for the average- or best-iterate, and (iii) with convergence rates into a ball around the near-optimal parameter at a slow subexponential rate. The main analysis approach in these works mimics the projected subgradient descent analysis and necessitate realizability conditions. On the other hand, we prove that Gauss-Newton dynamics (i) achieves last-iterate \textit{convergence} to an in-class optimum predictor, (ii) without any realizability assumptions, and (iii) at a fast geometric convergence rate, emphasizing the benefits of Gauss-Newton preconditoning in the underparameterized regime.

\subsection{Notation}
For a differentiable curve $\gamma:I\subset \bR^+ \rightarrow \bR$, $\dot{\gamma}_t$ and $\gamma'(t)$ denote its derivative at time $t$. $\bm{I}$ denotes the identity matrix. $\succcurlyeq$ is the Loewner order. For a smooth function $f:\bR^n\rightarrow\bR^p$, $\rr{Lip}_f$ denotes its modulus of Lipschitz continuity. For a symmetric positive-definite matrix $\bm{A}\in\bR^{n\times n}$ and $v\in\bR^n$, $\|x\|_{\bm{A}}^2:=x^\top \bm{A}x$. We define $\|v\|_2^2:=v^\top v$ for $v\in\bR^n$, and $\|v\|=\|v\|_2$ unless specified otherwise. For $h:\bR\rightarrow \bR$, $\|h\|_\infty:=\sup_{z\in\bR}|h(z)|$. For a matrix $\bm{P}\in\bR^{n\times n}$, $\|\bm{P}\|$ denotes its operator norm and $\lambda_{\min}(\bm{P})$ denotes its minimum eigenvalue. We denote the unit sphere in $\bR^n$ as $\bS^{n-1}:=\{x\in\bR^n:\|x\|_2=1\}$.

\section{Problem Setting and the Gauss-Newton Dynamics}
\subsection{Supervised Learning Setting}\label{subsec:sl}
In this work, we consider a supervised learning problem with a data set $$\mathcal{D}=\{(x_j,y_j)\in\bR^d\times\bR:1\leq j \leq n\},$$ where $\{x_j\in\bR^d:j\in[n]\}$ are the training inputs and $\{y_j\in\bR:j\in[n]\}$ are the outputs. Given a loss function $\ell:\bR\times\bR\rightarrow\bR_+$, the empirical risk for the prediction $\xi=(\xi_1~\ldots~\xi_n)$ is defined as
    $$g(\xi):=\sum_{j=1}^n\ell(\xi_j,y_j).$$

\textbf{Deep fully-connected neural networks.} We consider a feedforward deep neural network of depth $H \geq 1$ and width $m$ with a smooth (i.e., infinitely-differentiable) activation function $\sigma:\bR\rightarrow\bR$ such that $$\sup_{z\in\bR}|\sigma(z)|\leq \sigma_0,~\sup_{z\in\bR}|\sigma'(z)|\leq \sigma_1,\mbox{ and }\sup_{z\in\bR}|\sigma''(z)|\leq \sigma_2.$$ Note that many widely-used activation functions, including $\tanh$ (with $\sigma_0=1,\sigma_1=2,\sigma_2=2$) and sigmoid function, satisfy this.

Let $W^{(1)}\in\bR^{m\times d}$ and $W^{(h)}\in\bR^{m\times m}$ for $h=2,3,\ldots,H$, and $\bm{W}:=(W^{(1)},\ldots,W^{(H)})$. Then, given a training input $x_j\in\bR^d$, the neural network is defined recursively as
\begin{align}
\begin{aligned}
    \bm{x}_j^{(h)}(\bm{W})&=\sqrt{\frac{a_\sigma}{m}}\cdot\vec{\sigma}\Big(W^{(h)}\bm{x}^{(h-1)}_j(\bm{W})\Big),\quad h=1,2,\ldots,H,\\
    \varphi(x_j;w)&=c^\top\bm{x}_j^{(H)}(\bm{W})
    \end{aligned}
\end{align}
where $\bm{x}_j^{(0)}(\bm{W})=x_j$, $a_\sigma:=(\bE_{z\sim\cN(0,1)}[\sigma^2(z)])^{-1}$ is normalization parameter, $w={\rr{vec}}(\bm{W},c)$ is the parameter vector, and $\vec{\sigma}(z)=[\sigma(z_1)~\ldots~\sigma(z_m)]^\top$.

\textbf{Random initialization.} We adopt the standard NTK initialization \(w_0=(c_0,\bm W_0)\) as in \cite{ji2020polylogarithmic,du2018gradient}: for each layer $h\in[H]$,
\begin{equation}
    [c_{0}]_i\overset{\rr{iid}}{\sim}\rr{Rad}\mbox{ and }[W_{0}^{(h)}]_{ij}\overset{\rr{iid}}{\sim}\cN(0,1).
\end{equation}
We denote the prediction function as $$f(w):=[
    \varphi(x_1;w)
    ~\ldots
    ~\varphi(x_n;w)]-b,$$ where $b\in\bR^n$ is a fixed bias term. To ensure $f(w_0)=0$ following \cite{chizat2019lazy,mjt2021deep, bai2019beyond}, we set $b_j=\varphi(x_j;w_0)$ for each $j\in[n]$, which simply re-centers the model in function space. Since $b$ is constant in the learnable parameter $w$, gradients and Hessians and thus gradient-based dynamics remain unchanged.

\textbf{Optimization in supervised learning.} The objective in this paper is to empirical risk minimization:
\begin{equation}
    \min_{w\in\bR^p}~g(\alpha f(w))=:\cR(w),
\end{equation}
where $\alpha > 0$ is an output scaling parameter, which will be critical in the convergence results in this paper. We assume that $g:\bR^n\rightarrow\bR_+$ is (Euclidean) $\nu$-strongly convex and has $\mu$-Lipschitz continuous gradients as a function of the prediction $\xi\in\bR^n$:
\begin{equation}
    \nu \bm{I}\preccurlyeq \nabla^2g(\xi)\preccurlyeq\mu\bm{I},\quad\xi\in\bR^n.
\end{equation}
In the case of quadratic loss $g(\xi)=\frac{1}{2}\sum_{j=1}^n(\xi_j-y_j)^2$, we have $\mu=\nu=1$. Note that $w\mapsto g(\alpha f(w))=:\cR(w)$ is highly nonconvex, which constitutes the main challenge in the empirical risk minimization problem. The number of learnable parameters is $p$, i.e., $w\in\bR^p$. We call the model \emph{underparameterized} if $p\leq n$, and overparameterized otherwise. For each layer $h \in [H]$, we denote the Jacobians of $f(w)$ with respect to $W^{(h)}$ and $w$, respectively, as 
\begin{equation}
    \rr{D}_h f(w):=\begin{bmatrix}
        \nabla^\top_{\rr{vec}(W^{(h)})}\varphi(x_1;w)\\
        \vdots\\
        \nabla^\top_{\rr{vec}(W^{(h)})}\varphi(x_n;w)\\
    \end{bmatrix}
    \mbox{ and }\rr{D}f(w):=\begin{bmatrix}
    \nabla_w^\top \varphi(x_1;w)\\
    \vdots\\
    \nabla_w^\top \varphi(x_n;w)
    \end{bmatrix}.
\end{equation}

\subsection{Gauss-Newton Gradient Flow}
In this work, we consider Gauss-Newton gradient flow for training neural networks:
\begin{align}
    \begin{cases}
        \begin{aligned}
    \dt{w_t} &= -\frac{1}{\alpha}\big[\bm{H}_\rho(\alpha f(w_t))\big]^{-1}\rr{D}^\top f(w_t)\nabla g(\alpha f(w_t))~~~\mbox{for}~t > 0,\\
    w_0 &= w_\mathrm{init},    
    \end{aligned}
    \end{cases}
    \label{eqn:gauss-newton}
\end{align}
where $\alpha > 0$ is a scaling factor, and 
\begin{equation}
\bm{H}_\rho(w) := (1-\rho(w))\rr{D}^\top f(w)\nabla_f^2g(\alpha f(w))\rr{D}f(w)+\rho(w) \bm{I}
\label{eqn:preconditioner}
\end{equation}
is the preconditioner with the regularization (or damping) factor $\rho:\bR^p\rightarrow[0,1].$
In case $\rr{D}^\top f(w_t)\nabla_f^2g(\alpha f(w_t))\rr{D} f(w_t)$ is singular, which is the case in overparameterized problems with $n > p$, regularization is used to ensure that $\bm{H}_\rho(w_t)$ is non-singular. The case $\rho(w) > 0$ is known as the Levenberg-Marquardt dynamics \citep{nocedal1999numerical}.

\section{Gauss-Newton Dynamics for Overparameterized Neural Networks}\label{sec:op}
We start with the analysis in the overparameterized regime with $p > n$. Since we have $\rank\big(\rr{D}^\top f(w)\rr{D}f(w)\big) \leq n < p$ in this regime, we consider regularization (or damping) $\rho > 0$ to ensure that \eqref{eqn:gauss-newton} is well-defined, which leads to the Levenberg-Marquardt dynamics. As the analysis will indicate, the regularization schedule $\rho$ plays a fundamental role on the convergence of Gauss-Newton dynamics in the overparameterized regime, and we propose provably good regularization schemes to achieve fast convergence without any dependence on the spectrum of $\bm{K}_0$ and (locally) on $\nabla^2\cR(w)$. The proof in the overparameterized regime extends the kernel analysis in \cite{chizat2019lazy, du2018gradient} for the gradient flows to the Gauss-Newton gradient flows, and setting $\rho(w)=1$ in our theoretical results will recover the existing bounds. 

We first present the analysis of the case \(H=1\) with the output layer $c$ frozen at initialization, which conveys the main ideas with sharp bounds and minimal notation. Hence, $w=\rr{vec}(W^{(1)})$ and $p = md$ here. Corollary \ref{cor:deep} extends the arguments to deep networks with trainable output layers.

In the overparameterized regime, the spectral properties of the so-called neural tangent kernel has a crucial impact on the convergence. To that end, let $\bm{K},\bar{\bm{K}}\in\bR^{n\times n}$ be defined as 
\begin{align*}
[\bar{\bm{K}}]_{ij} &:= x_i^\top x_j\bE_{u_0\sim\cN(0,\bm{I}_d)}[\sigma'(u_0^\top x_i)\sigma'(u_0^\top x_j)]a_\sigma^2 ,\quad i,j\in\{1,2,\ldots,n\},\\
\bm{K}(w)&:=\rr{D}f(w)\rr{D}^\top f(w),\quad w\in\bR^p.
\end{align*} Note that under the initialization $(c,w_\rr{init})$, we have $\bE[\bm{K}(w_\rr{init})] = \bar{\bm{K}}.$ We make the following standard representational assumption on the so-called neural tangent kernel evaluated at $\mathcal{D}$ \citep{chizat2019lazy}.
\begin{assumption}
    Assume that $\bm{K}(w_0)$ is strictly positive definite with the minimum eigenvalue $4\lambda^2 > 0$.
    \label{assumption:ntk}
\end{assumption}
\begin{remark}[Conditioning of the neural tangent kernel matrix $\bm{K}_0$]\normalfont
    The geometry of the data points $\{x_i\in\bR^d:i=1,2,\ldots,n\}$ has a significant impact on the spectrum of $\bm{K}_0$, thus $\lambda^2$. If the data points are uniformly distributed on $\bS^{d-1}$ for $d \geq 2$ as $x_i\sim_\rr{iid}\rr{Unif}(\bS^{d-1})$, then we have (up to logarithmic factors) $$n^{-\frac{4}{d-1}}\lesssim \lambda^2 \lesssim n^{-\frac{2}{d-1}}$$ with high probability, while we have $\lambda^2 \lesssim \delta'(\cD):=\min_{i\neq j}\|x_i-x_j\|_2$ more generally \citep{karhadkar2024bounds}. As such, while $\lambda >0$ holds in general, $\bm{K}_0$ can be highly ill-conditioned for large training sets $\cD$, implying a very small $\lambda^2$. Since the convergence rate of the gradient flow is $\exp(-\nu\lambda^2 t)$ \citep{chizat2019lazy, du2018gradient}, small $\lambda^2\approx 0$ implies an arbitrarily slow convergence.
\label{remark:geometry}
\end{remark}

\noindent For a single-hidden layer neural network ($H=1$), the prediction function $w\mapsto \alpha f(w)$ has globally $L$-Lipschitz gradients (i.e., $g$ is $L$-smooth) with \begin{equation}\label{eqn:L}L=\frac{\sigma_2}{\sqrt{m}}\sqrt{\sum_{j=1}^n\|x_j\|_2^4}.\end{equation} Under Assumption \ref{assumption:ntk}, if \begin{equation}\label{eqn:r0-op}\|w-w_\rr{init}\|_2 < r_0:=\frac{\lambda}{L},\end{equation} then $\rr{D}f(w)\rr{D}^\top f(w) \succcurlyeq \lambda^2 \bm{I}$ \citep{mjt2021deep, chizat2019lazy}.

Under the Gauss-Newton gradient flow \eqref{eqn:gauss-newton}, define the exit time $$T:=\inf\{t > 0:\|w_t-w_0\|_2 \geq  r_0\}.$$ Also, let $\bm{K}_t :=\bm{K}(w_t),~t\in[0,\infty)$ be the kernel matrix, and $\lambda_t^2:=\lambda_{\min}(\bm{K}_t)$. Then, we have
\begin{align}
    \inf_{t\in[0,T)}\lambda_t^2 \geq \lambda^2.
    \label{eqn:ntk-bound}
\end{align}

We start with the analysis with the Gauss-Newton preconditioner derived the quadratic loss function, and we extend the results to non-quadratic losses in Section \ref{subsec:non-quad}.

\subsection{Gauss–Newton in the Overparameterized Regime}
We first consider the preconditioner $$\bm{H}_\rho(w_t)=(1-\rho_t)\rr{D}^\top f(w_t)\rr{D}f(w_t) + \rho_t\bm{I},$$ under various regularization schemes $\rho_t:=\rho(w_t)\in(0,1]$. We note that the above preconditioner approximates $\nabla_w^2g$ where $g(z)=\frac{1}{2}\|z-y\|_2^2$. In the analysis, we evaluate the performance of this preconditioner for general $\nu$-strongly-convex and $\mu$-smooth $g$. For quadratic loss $g(z)=\frac{1}{2}\|z-y\|_2^2$, we have $\mu=\nu=1$.

The gradient flow in the function space and the energy dissipation inequality \eqref{eqn:evi} under any damping scheme $\rho_t > 0$ in this regime are presented in the following lemma.
\begin{lemma}
    Under the Gauss-Newton gradient flow with any $(\rho_t)_{t\in[0,\infty)}$ such that $\rho_t\in(0,1]$. Then,
    \begin{align}
        \tag{GF-O} \dt{\alpha f(w_t)} &= -\frac{1}{\rho_t}\left(\bm{K}_t - \frac{1-\rho_t}{\rho_t}\bm{K}_t\left(\bm{I}+\frac{1-\rho_t}{\rho_t}\bm{K}_t\right)^{-1}\bm{K}_t\right)\nabla g(\alpha f(w_t)),\label{eqn:ei}\\
        \tag{EDI} \dt{g(\alpha f(w_t))} &\leq \frac{-\lambda_t^2}{\rho_t+(1-\rho_t)\lambda_t^2} \|\nabla g(\alpha f(w_t))\|_2^2,\label{eqn:evi}
    \end{align}
    for any $t < T$.
    \label{lemma:si-evi}
\end{lemma}
\noindent The proof of Lemma \ref{lemma:si-evi} follows from the Sherman–Morrison–Woodbury matrix identity \citep{horn2012matrix}, and can be found in Appendix \ref{app:op}.

The damping scheme $(\rho_t)_{t\in\bR^+}$ has a pivotal role on the convergence of Gauss-Newton in the overparameterized regime. In the following, we establish finite-time convergence bounds for the Gauss-Newton dynamics under constant and an adaptive damping schemes.

\subsubsection{Convergence of Gauss-Newton under Stationary Damping}
    Note that Lemma \ref{lemma:si-evi} implies $\cR(w_t)$ is monotonically decreasing for any $t\in\bR^+$. In the following, we characterize the decay rate of the optimality gap for $t < T$.

    \begin{lemma}
        Under a stationary damping scheme $\rho_t=\rho\in(0,1]$, we have
    \begin{align}
    \begin{aligned}
        V_t &\leq V_0\exp\left(\frac{-2\nu \lambda^2t}{\rho+(1-\rho)\lambda^2}\right),\\
        \|\alpha f(w_t)-f^\star\|_2^2 &\leq \frac{2V_0}{\nu}\exp\left(\frac{-2\nu t\lambda^2}{\rho+(1-\rho)\lambda^2}\right),
        \end{aligned}
        \label{eqn:edi-constant}
    \end{align}
    for any $t \in [0,T)$, where $$V_t := g(\alpha f(w_t))-g(f^\star)$$ is the optimality gap, and $f^\star$ is the unique global minimizer of $g$ in $\bR^n$.
    \label{lemma:edi}
    \end{lemma}

    \noindent The proof of Lemma \ref{lemma:edi} can be found in Appendix \ref{app:op}. The finite-time error bounds in \eqref{eqn:edi-constant} motivate a class of provably effective regularization schemes that guarantee fast convergence rates independent of $\lambda_{\min}(\bm{K}_0)$, which we will explicitly characterize next. 
    
    Note that the finite-time bounds in Lemma \ref{lemma:edi} hold for $t\in[0,T)$. In the following, we prove that the first-exit time $T=\infty$, which implies that the kernel $\bm{K}_t$ is non-degenerate for any $t>0$, if the scaling factor $\alpha\sqrt{m}>0$ is sufficiently large, which implies convergence to the (globally optimal) empirical risk minimizer $f^\star$.

    \begin{lemma}[Kernel non-degeneracy]
    Consider the Gauss-Newton dynamics with any constant damping scheme $\rho_t=\rho \in \big(0,\frac{\lambda^2}{1+\lambda^2}\big]$ for $t\geq 0$. If $\alpha \geq\frac{\mu L\sqrt{2V_0}}{\nu^{3/2}}\frac{1}{\lambda^2},$ then $T=\infty$.
\label{lemma:non-degeneracy}
\end{lemma}
\noindent Using \eqref{eqn:L}, the sufficient condition for $T=\infty$ is $$\alpha\sqrt{m} \geq\frac{\mu \sigma_2\sqrt{2g(0)\sum_{j=1}^n\|x_j\|_2^4}}{\nu^{3/2}}\frac{1}{\lambda^2}.$$ This leads us to the following convergence result for the Gauss-Newton gradient flow. 

    \begin{theorem}[Convergence in the overparameterized regime] The Gauss-Newton gradient flow \eqref{eqn:gauss-newton} with a constant damping factor $\rho_t=\rho \in (0,\frac{\lambda^2}{1+\lambda^2}],~t\in[0,\infty)$ yields the following finite-time bounds under Assumption \ref{assumption:ntk}:
\begin{align}
\begin{aligned}
    V_t &\leq V_0\cdot\exp\left(-\frac{2\nu t\lambda^2}{\rho+(1-\rho)\lambda^2}\right),\\
    \|\alpha f(w_t)-f^\star\|_2^2 &\leq \frac{\mu}{\nu}\cdot\|f^\star\|_2^2\cdot \exp\left(-\frac{2\nu t\lambda^2}{(1-\rho)\lambda^2+\rho}\right)
    \end{aligned}
    \label{eqn:delta-op}
\end{align}
for any $t \in\bR^+$ with the scaling factor 
$\alpha \geq\frac{\mu L\sqrt{V_0}}{\nu^{3/2}}\frac{1}{\lambda^2}.$
\label{thm:lm-convergence}
\end{theorem}

\noindent Note that setting $\rho = \frac{\lambda^2}{1+\lambda^2}$ in \eqref{eqn:delta-op} yields $$V_t \leq V_0 \exp(-\nu (1+\lambda^2) t)\leq g(0)e^{-\nu t}$$ for any $t\geq 0$, which implies a convergence rate independent of $\lambda^2$.

\begin{remark}[On the benefits of preconditioning]\label{remark:overparameterized}\normalfont
    The gradient flow achieves a convergence rate $\exp(-2\nu\lambda^2 t)$ \citep{chizat2019lazy}. As such, a small $\lambda$, which frequently occurs in practice (see Remark \ref{remark:geometry}), implies arbitrarily slow convergence for the gradient flow. On the other hand, with the choice $\rho_t=\frac{\lambda^2}{1+\lambda^2}$ for $t \geq 0$, the convergence rate becomes $\exp\left(-\nu t\right)$, which is \emph{independent} of $\lambda$. This indicates that preconditioning by $\Ht$ in the Gauss-Newton method yields fast convergence even when the kernel $\bm{K}_0$ is ill-conditioned.
\end{remark}

The data-dependent damping choice $\rho = \frac{\lambda^2}{1+\lambda^2}$ yields geometric convergence rate \emph{independent} of $\lambda^2$, implying that the accelerated convergence rate does not stem from time-scaling in continuous time, and it is inherent to Gauss-Newton.

\begin{theorem}[Convergence in the overparameterized regime -- discrete time] Let $$h_i(z) := \frac{z^2}{\big((1-\rho)z^2+\rho\big)^i},\quad i=1,2,$$ and $\alpha = 1$. Consider the following discrete-time regularized Gauss-Newton method in discrete time:
\begin{align}\tag{GN-DT}
\begin{aligned}
    w_{k+1}&=w_k - \eta \left[(1-\rho)\rr{D}^\top f(w_k)\rr{D} f(w_k)+\rho\bm{I}\right]^{-1}\rr{D}^\top f(w_k)\nabla g( f(w_k)),\\
    w_0&=w_\rr{init},
\end{aligned}
\end{align}
for $k\in\bN$. Under Assumption \ref{assumption:ntk}, the Gauss-Newton method with the damping factor $\rho \in \big(0,\frac{\lambda^2}{1+\lambda^2}\big]$ and the learning rate $\eta\leq \frac{h_1(\lambda)}{6 h_1^2(\rr{Lip}_f)\mu}$ yields
\begin{equation}
    V_k \leq V_0 \left(1-\eta\nu h_1(\lambda)\right)^k \mbox{ for any }k\in\bN,
\end{equation}
for $m\in\bN$ sufficiently large so that
\begin{align*}
    \frac{\sigma_2\sqrt{\sum_{j=1}^n\|x_j\|_2^4}}{\sqrt{m}}\leq \sqrt{\frac{\nu}{V_0}}\min\left\{\frac{h_1(\rr{Lip}_f)}{h_2(\lambda)\mu\eta},\frac{h_1^2(\rr{Lip}_f)}{h_2(\lambda)},\frac{\lambda\nu h_1(\lambda)}{2\sqrt{2h_2(\lambda)}}\right\},
\end{align*}
where $V_k:= g(\alpha f(w_k))-\inf_{z\in\bR^n}g(z)$. Setting $\eta = \frac{h_1(\lambda)}{6 h_1^2(\rr{Lip}_f)\mu}$ and $\rho=\frac{\lambda^2}{1+\lambda^2}$ yields $$V_k\leq V_0\left(1-\frac{1}{24}\cdot\frac{\nu}{\mu}\cdot \frac{(1+\lambda^2)^2}{h_1^2(\rr{Lip}_f)}\right)^k\leq V_0\left(1-\frac{1}{24}\cdot\frac{\nu}{\mu}\cdot \frac{1}{h_1^2(\rr{Lip}_f)}\right)^k,~k\in\bN,$$ which is independent of $\lambda$.
    \label{thm:lm-convergence-dt}
\end{theorem}
\noindent The proof of Theorem \ref{thm:lm-convergence-dt} can be found in Appendix \ref{app:op}.

 In the following, we consider a specific adaptive damping scheme that interpolates between the Gauss-Newton method and the gradient flow depending on the conditioning of the kernel $\bm{K}_t$.

    \subsubsection{Convergence of Gauss-Newton under Adaptive Damping}
    Recall that $\lambda_t^2$ is the minimum eigenvalue of $\bm{K}_t=\rr{D} f(w_t)\rr{D}^\top f(w_t)$. Define $(\rho_t)_{t\in[0,T)}$ as
    \begin{equation}
    \rho_t:=\frac{a\lambda_t^2}{1+a\lambda_t^2},\quad \mbox{for any }t\in[0,T),
    \label{eqn:adaptive-damping}
    \end{equation}
    where $a >0$ is a design parameter. We call this choice $(\rho_t)_{t\geq 0}$ the adaptive damping scheme. Note that $\rho_t > 0$ for all $t\in[0,T)$, thus the preconditioner is invertible and the differential equation in \eqref{eqn:gauss-newton} is well-defined for $t < T$.
\begin{remark}[Hybrid first- and second-order optimizers via adaptive $\rho_t$]\normalfont
     Regularization interpolates between the gradient flow and Gauss-Newton \citep{nocedal1999numerical}. $\rho_t=\frac{\lambda_t^2}{1+\lambda_t^2}$ performs this hybridization in an adaptive way depending on the spectrum of $\bm{K}_t:=\rr{D}f(w_t)\rr{D}^\top f(w_t)$:
     \begin{itemize}
         \item For ill-conditioned $\bm{K}_t$, the gradient flow has a slow convergence rate \cite{chizat2019lazy, du2018gradient}, thus the weight of the GN preconditioner $\rr{D}^\top f(w_t)\rr{D}f(w_t)$ increases for mitigation.
         \item First-order optimization achieves fast convergence for well-conditioned $\bm{K}_t$, thus $\rho_t\lessapprox 1$ in that case.
     \end{itemize}
     The impact of such an adaptive choice of $\rho_t$ is rigorously characterized in Corollary \ref{cor:lm-convergence-ada}.
\end{remark}

\begin{corollary}[Convergence under adaptive damping] The Gauss-Newton gradient flow with the regularization schedule $$\rho(w)=\frac{a\lambda_{\min}(\bm{K}(w))}{1+a\lambda_{\min}(\bm{K}(w))}$$ with any design choice $a > 0$ yields
    \begin{align}
        \begin{aligned}
        V_t &\leq V_0\cdot \exp\left(-2\nu\frac{1+a\lambda^2}{1+a}t\right),\\
        \end{aligned}
        \label{eqn:delta-op-adaptive}
    \end{align}
    for any $t\in\bR^+$ with the scaling factor
    $
        \alpha \geq \frac{\mu L\sqrt{2V_0}}{\nu^{3/2}}\cdot \frac{1+\lambda\rr{Lip}_f}{\lambda^2(1+\lambda^2)}.
    $
\label{cor:lm-convergence-ada}
\end{corollary}
\noindent The proof of Corollary \ref{cor:lm-convergence-ada} follows from a similar logic as Theorem \ref{thm:lm-convergence}, and thus omitted.

\subsection{General Loss Functions and Fast Local Convergence}\label{subsec:non-quad}
In this section, we study the convergence of the Gauss-Newton dynamics for general smooth and strongly convex $g:\bR^n\rightarrow \bR$. For simplicity, suppose that $g(f^\star) = 0$, which always holds by shifting the loss function if necessary.

Let
$$
\bm{H}_\rho(w) = (1-\rho(w))\rr{D}^\top f(w)\bm{G}(f(w))\rr{D}f(w) + \rho(w)\bm{I},
$$
for $\rho:\bR^p\rightarrow  (0, 1)$. Note that, in the quadratic case, the preconditioner $\bm{H}_\rho$ reduces to \eqref{eqn:preconditioner} since $\bm{G}( f(w)) = \bm{I}$ for any $w\in\bR^p$. We define two curvature-aware damping schemes as
\begin{align}
    \rho^{(1)}(w) &:= \frac{\lambda^2\cdot \lambda_{\max}(\bm{G}(f(w)))}{1+\lambda^2\cdot\lambda_{\max}(\bm{G}(f(w)))} \label{eqn:rho-out}\\
    \rho^{(2)}(w) &:= \frac{\sqrt{\cR(w)}}{c+\sqrt{\cR(w)}},\label{eqn:rho-out-2}
\end{align}
where $c = \rr{Lip}_f\sqrt{\mu/2}$.

We have the following result on the convergence of the Gauss-Newton dynamics.
\begin{theorem}\label{thm:non-quad-out}
    Under Assumption \ref{assumption:ntk}, 
    for $$\sqrt{m} \geq \frac{3\sigma_2}{\sqrt{\nu}}\frac{\lambda^2+\nu^{-1}}{\lambda^2+\mu^{-1}}\frac{\mu}{\lambda^2}\sqrt{\sum_{j}\|x_j\|_2^4},$$ the Gauss-Newton dynamics with the damping schedule $\rho^{(1)}$ yields
    \begin{align*}
        V_t &\leq V_0\cdot \exp\left(-\nu\lambda^2 t-\nu\int_{0}^t\lambda_{\max}^{-1}(\bm{G}(s))\rr{d}s\right),\quad t\in[0,\infty).
    \end{align*}
    Furthermore, there exists a global minimizer $w^\star\in\cB(w_0,r_0)$ such that 
    \begin{equation}
    \lim_{t\rightarrow\infty}w_t=w^\star\quad\mbox{and}\quad f(w^\star)\in\arg\min_{z\in\bR^n}~g(z).
    \label{eqn:par-convergence}
    \end{equation}
    Let $T_r:=\inf\{t>0:\|w_t-w^\star\|\leq r\}$ for $r > 0$. Assume that $w\mapsto\nabla^2\cR(w)$ is $C$-Lipschitz continuous. The, there exists $r^\star :=r^\star(C, \nabla^2\cR(w^\star))$ such that the Gauss-Newton dynamics under $\rho^{(2)}(w_t)$ for $t\in[T_{r^\star},\infty)$ yields 
    $$
        \frac{1}{2}\|w_t-w^\star\|_2^2 \leq \frac{1}{2}\|w_{T_{r^\star}}-w^\star\|_2^2\cdot e^{-(t-T_{r^\star})},\quad t \in [T_{r^\star},\infty).
    $$
\end{theorem}

\noindent Theorem \ref{thm:non-quad-out} implies that the adaptive damping schedule in \eqref{eqn:rho-out} yields convergence to the empirical minimizer in the predictor space. Furthermore, the adaptive damping schedule \eqref{eqn:rho-out} yields convergence to an optimal parameter \eqref{eqn:par-convergence}. Upon entering a certain neighborhood of the optimal parameter $w^\star$, which is ensured by the first part, switching to the damping schedule $\rho^{(2)}$ lead to fast convergence independent of $\kappa$ and $\lambda^2$. As such, the Gauss-Newton method achieves fast convergence independent of the conditioning of $\nabla^2g$ and $\bm{K}$.

    \begin{remark}[Implicit bias of Gauss-Newton for quadratic loss]
            Equation \eqref{eqn:par-convergence} implies that the parameter $w_t$ under Gauss-Newton dynamics converges to an empirical risk minimizer. Since $f\mapsto g(f)$ is strongly convex, there exists a unique minimizer $f^\star\in\bR^n$ predictor space; however there may be many empirical minimizers in parameter space that yields $f(w^\star) = f^\star$. An important question in deep learning is to characterize which minimizer $w^\star$ is chosen by the training algorithm, which is known as the implicit bias of the particular learning algorithm. The following result extends the implicit bias characterization of gradient descent in Section 12.1.1 in \cite{bach2024learning} to the Gauss-Newton method for quadratic loss function $g(f)=\frac{1}{2}\|f-y\|_2^2$.

            \begin{proposition}[Implicit bias of Gauss-Newton for quadratic loss]\label{prop:implicit-bias}
    Under Assumption \ref{assumption:ntk}, for any $u\in\bR^p$, let $\bar{f}_0(u)=\rr{D}f(w_0)(u-w_0)$ and consider $$\dot u_t = -\bm{H}^{-1}_\rho(w_0)\rr{D}^\top f(w_0)(f(u_t)-y),\quad t \geq 0,$$ with $u_0=w_0$ for $\rho \in (0,\frac{\lambda^2}{1+\lambda^2}]$. Then, $u_t\rightarrow u^\star$ as $t\rightarrow\infty$, where $u^\star$ is the solution of
        $$
            \min_{u\in\mathbb{R}^p}~\|u-w_0\|_{\bm{H}_\rho(w_0)}^2\quad \mbox{s.t.}\quad \bar f_0(u)=y.
        $$
\end{proposition}

            \noindent We provide a proof in Appendix \ref{app:op}. By Theorem 2.2 in \cite{chizat2019lazy}, this implies that $\|w^\star-u^\star\|_2 = \cO(1/\alpha^2)$. As such, Gauss-Newton converges into $\cO(1/\alpha^2)$-neighborhood of the minimum-$\bm{H}_\rho(w_0)$ norm interpolant in the kernel regime.
    \end{remark}

\subsection{Convergence of Gauss-Newton for Deep Neural Networks}\label{sec:op-deep}
In this section, we extend the analysis to deep neural networks, where the output layer $c$ is also trained. The absence of global Lipschitz-smoothness of $w\mapsto f(w)$ constitutes the main challenge. We address this challenge by leveraging tools from \cite{du2019gradient} with improved bounds in terms of the sample size $n$ and the confidence $\delta\in(0,1)$. In this subsection, we assume that $\bm{x}_i\in\mathbb{S}^{d-1},~i\in[n]$ and $$\bm{K}^{(H)}(w_0) \succeq 4\lambda^2\bm{I}$$ for some $\lambda > 0$. By Lemma B.2 in \cite{du2019gradient}, $\bm{x}_i\nparallel\bm{x}_j,~i\neq j$ implies that this assumption is satisfied for sufficiently large $m$. Let $\bm{K}^{(h)}(w):=\rr{D}_hf(w)\rr{D}_h^\top f(w)$. Then, $$\bm{K}(w):=\sum_{h=1}^H\bm{K}^{(h)}(w)\succcurlyeq\bm{K}^{(H)}(w),\quad w\in\bR^p.$$ The following lemma indicates that there exists a neighborhood $B'$ of $w_0$ such that $\lambda_{\min}(\bm{K}(w)) > 0$ for any $w\in B'$.


\begin{lemma}\label{lemma:par-dev-deep}
    Let $C=2\sigma_0\sigma_1a_\sigma$. For any $\delta \in (0,1)$, let
    \begin{equation}\label{eqn:m-deep}
        m \geq \sigma_0^4a_\sigma^2\left(\frac{C^H-1}{2(C-1)}\right)^2\log\Big(\frac{2Hn}{\delta}\Big),
    \end{equation}
    and $k_w:=1+2\sqrt{\max\left\{1,d/m\right\}}$. Then, with probability at least $1-\delta - 2He^{-m/2}$ over the random initialization, if $$\|w-w_0\|_2 \leq k\frac{k_x-1}{k_x^H-1}\min\left\{1,\lambda_{\min}(\bm{K}^{(H)}_0)/n\right\}\sqrt{m}=:R\sqrt{m},$$ for a universal constant $k > 0$, we have $$ \lambda_{\min}(\bm{K}(w)) \geq \lambda_{\min}(\bm{K}^{(H)}(w)) \geq \lambda_{\min}(\bm{K}^{(H)}(w_0))/4.$$
\end{lemma}

\noindent Lemma \ref{lemma:par-dev-deep} builds on Lemma B.4 in \cite{du2019gradient} with slightly improved dependencies on $n$ and $\delta$, and its proof can be found in Appendix \ref{app:op}. Using Lemma \ref{lemma:par-dev-deep}, we obtain the following convergence result.

    \begin{corollary}[Convergence of Gauss-Newton for deep networks]\label{cor:deep}
        For any $\delta \in (0, 1)$, if the neural network width $m$ is chosen sufficiently large such that \eqref{eqn:m-deep} holds, then the Gauss-Newton gradient flow with a constant damping $\rho \in (0, \frac{\lambda^2}{1+\lambda^2}]$ and $\alpha\sqrt{m}\geq \frac{\mu\sqrt{2V_0}}{\lambda R \nu^{3/2}}$ yields
        $$
        V_t \leq V_0\exp\Big(-\frac{2\nu\lambda^2t}{\rho + (1-\rho)\lambda^2}\Big),
        $$
        for any $t \geq 0$ with probability at least $1-\delta - 2H\exp(-m/2)$ over the random initialization.
    \end{corollary}

\section{Gauss-Newton Dynamics for Underparameterized Neural Networks: Riemannian Optimization}\label{sec:up}
In the underparameterized regime characterized by $p < n$, the kernel $\rr{D}f(w_t)\rr{D}^\top f(w_t)\in\bR^{n\times n}$ is singular for all $t \in [0,\infty)$ since $\rr{rank}(\rr{D}f(w)\rr{D}^\top f(w)) \leq p<n$ for all $w\in\bR^p$. Thus, the analysis in the preceding section, which relies on the non-singularity of $\bm{K}_t$, will not extend to this setting. This will motivate us to study the underparameterized regime by using tools from optimization on Riemannian manifolds.

In the underparameterized regime, the Gauss-Newton preconditioner $\bm{H}_\rho(\alpha f(w))\in\bR^{p\times p}$ can be non-singular without damping (i.e., $\rho=0$) since $p < n$. Thus, we consider Gauss-Newton dynamics without damping. The non-singularity of $\Ht|_{\rho=0}$ will be crucial in establishing the Riemannian optimization framework in the succeeding sections. For a detailed discussion on optimization on embedded submanifolds, which is the main toolbox in this section, we refer to \cite{boumal2023introduction, absil2008optimization, udriste2013convex}. 
\begin{assumption}
    Let $\bm{H}_0 :=  \rr{D}^\top f(w_\rr{init})\rr{D} f(w_\rr{init})$. There exists $\lambda_0 > 0$ such that $\bm{H}_0 \succcurlyeq 4\lambda_0^2\bm{I}$.
    \label{assumption:gram}
\end{assumption}
We define $$B := \Big\{w\in\bR^p:\|w-w_\rr{init}\|_2\leq r_0\Big\},$$ where \begin{equation}\label{eqn:r0}r_0 = \min\left\{\frac{\lambda_0}{L}, \frac{1}{4}\cdot\frac{\lambda_0^2\nu}{\mu L\rr{Lip}_f} \right\}.\end{equation} The following result implies the non-degeneracy of the Gram matrix $\bm{H}_0(\alpha f(w))$ on $B$.
\begin{lemma}[$w\mapsto \alpha f(w)$ is an immersion on $B$]\label{lemma:min-eig}
    For any $w\in B$, we have $$\bm{H}_0(\alpha f(w)) = \rr{D}^\top f(w)\rr{D}f(w) \succcurlyeq \lambda_0^2\bm{I},$$ which implies that $\rank(\rr{D}f(w))=p$ for all $w\in B$. Thus, $f$ is an immersion on $B$.
\end{lemma}
\noindent The result follows from \cite{chizat2019lazy}, and we provide the proof in Appendix \ref{app:up} for completeness.

Let $$T := \inf\{t>0:w_t \notin B\}$$ be the first-exit time of $B$. Then, the Gauss-Newton gradient flow is well-defined for $t < T$ since we have a non-degenerate preconditioner with $\inf_{t<T}\lambda_{\min}(\bm{H}_0(\alpha f(w_t))) \geq \lambda_0$ and a full-rank $\rr{D}f(w_t)$ for $t < T$ by Lemma \ref{lemma:min-eig}.

We first characterize the gradient flow in the output space and the energy dissipation inequality in the underparameterized regime, following Section \ref{sec:op} and \cite{chizat2019lazy}.

\begin{lemma}
    For any $w \in B$, let
    \begin{equation}
        \bm{P}(\alpha f(w)) := \rr{D}f(w)\big(\bm{H}_0(\alpha f(w))\big)^{-1}\rr{D}^\top f(w).
        \label{eqn:projection}
    \end{equation}
    Then, for any $t<T$,
    \begin{align}
        \tag{GF-Ou}\dt{\alpha f(w_t)} &= -\bm{P}(\alpha f(w_t))\nabla g(\alpha f(w_t)),\label{eqn:ei-u}\\
        \tag{EDI-u}\dt{g(\alpha f(w_t))} &= - \|\bm{P}(\alpha f(w_t))\nabla g(\alpha f(w_t))\|_2^2.\label{eqn:evi-u}
    \end{align}
    \label{lemma:si-evi-u}
\end{lemma}
\begin{proof}
    By the chain rule, we obtain \eqref{eqn:ei-u} from $\dt{\alpha f(w_t)}=\alpha\rr{D}f(w_t)\dot{w}_t$ by substituting the dynamics in \eqref{eqn:gauss-newton}. \eqref{eqn:evi-u} is obtained from \eqref{eqn:ei-u} by using the fact that $\bm{P}(\alpha f(w_t))$ is idempotent.
\end{proof}
Note that $\rr{rank}(\bm{P}(\alpha f(w_t))) = p < n$ and $\bm{P}(\alpha f(w_t))$ is symmetric and idempotent, which implies that it is an orthogonal projection matrix for $t < T$. Since the minimum eigenvalue of $\bm{P}^2(\alpha f(w_t))=\bm{P}(\alpha f(w_t))$ is 0, \eqref{eqn:evi-u} only implies that $t\mapsto g(\alpha f(w_t))$ is a non-increasing function, which does not provide any useful information about the convergence rate or the optimality of the limiting predictor of the Gauss-Newton dynamics in the underparameterized setting. This motivates us to cast the problem as an optimization problem on a Riemannian manifold.

\subsection{Gauss-Newton Dynamics as a Riemannian Gradient Flow in the Output Space}
An immediate question on studying \eqref{eqn:ei-u} is the characterization of the subspace that $\bm{P}(\alpha f(w_t))$ projects the Euclidean gradient $\nabla g(\alpha f(w_t))$ onto. This motivates us to depart from the Euclidean geometry and study the output space $\alpha f(B)$ as a smooth submanifold.

For any $\alpha > 0$, let 
\begin{equation}
    \mathcal{M} := \alpha f(B) := \{\alpha f(w): w\in B\}.
\end{equation}

\noindent Note that $\bm{H}_0(\alpha f(w))$ is full-rank if $w\in B$, which implies that $\alpha \rr{D}f(w)$ is also full-rank. This has an important consequence.
\begin{lemma}\label{lemma:injection}
    $\alpha f|_B:B\rightarrow \mathcal{M}$ is an injective function, hence it is a smooth embedding on $B$.
\end{lemma}
\begin{proof}
    Consider two arbitrary points $w, w'\in B$, and let $$\gamma(t) := t w + (1-t)w',~t\in[0,1].$$ Then, we have $\dt{ f(\gamma(t))} =  \rr{D}f(\gamma(t))(w_1-w_2)$ for any $t\in[0,1],$ which implies that 
    \begin{align*}f(w)-f(w') &= \int_0^1\rr{D}f(\gamma(s))(w-w')\rr{d}s\\
    &= \rr{D}f(w_0)(w-w') + \int_0^1\left(\rr{D}f(\gamma(s))-\rr{D}f(w_0)\right)(w-w')\rr{d}s.
    \end{align*}
    First, note that 
    \begin{equation}
        \|\rr{D}f(w_0)(w-w')\|_2 = \sqrt{(w-w')\rr{D}^\top f(w_0)\rr{D}f(w_0)(w-w')} \geq 2\lambda\|w-w'\|_2,
        \label{eqn:inj-i}
    \end{equation}
    by Assumption \ref{assumption:gram}. Then, for any $s\in[0,1]$, $L$-Lipschitz continuity of $w\mapsto \rr{D}f(w)$ (where $L$ is given in \eqref{eqn:L}) implies 
    \begin{align}
        \nonumber \|(\rr{D}f(\gamma(s))-\rr{D}f(w_0))(w-w')\|_2 &\leq L\|\gamma(s)-w_0\|_2\|w-w'\|_2\\
        \nonumber &\leq L(s\|w-w_0\|_2 + (1-s)\|w'-w_0\|_2)\|w-w'\|_2\\
        &\leq \lambda \|w-w'\|_2, \label{eqn:inj-ii}
    \end{align}
    since $w',w\in B$. Using \eqref{eqn:inj-i} and \eqref{eqn:inj-ii}, we obtain
    \begin{align*}
        \|f(w)-f(w')\|_2 &\geq \|\rr{D}f(w_0)(w-w')\|_2 - \int_0^1\left\|\left(\rr{D}f(\gamma(s))-\rr{D}f(w_0)\right)(w-w')\right\|_2\rr{d}s\\
        &\geq \lambda \|w-w'\|_2.
    \end{align*}
    Hence, if $w\neq w'$, then $f(w)\neq f(w')$, which implies that $w\mapsto f(w)$ is injective on $B$. Recall that $\alpha f|_B$ is a smooth immersion on $B$ by Lemma \ref{lemma:min-eig}. Note that $B\subset \bR^p$ is a closed ball in $\bR^p$, which implies that it is a compact smooth manifold with boundary. Since $\alpha f|_B$ is an injective smooth immersion and $B$ is a compact smooth manifold with boundary, Proposition 4.22 in \cite{lee2012introduction} implies that $\alpha f|_B$ is a smooth embedding.
\end{proof}

The following result shows that the function space $\cM$ is a smooth embedded submanifold of the Euclidean space $\bR^n$.
\begin{proposition}
    $\cM$ is a $p$-dimensional smooth embedded submanifold of $\bR^n$ with boundary.
    \label{prop:submanifold}
\end{proposition}
\begin{proof}
    Since $B$ is a closed ball in $\bR^p$, it is a smooth manifold with boundary. Also, Lemma \ref{lemma:injection} implies that $\alpha f:B\rightarrow\bR^n$ is a smooth embedding. Hence, $\mathcal{M}:=\alpha f(B)$ is a $p$-dimensional embedded smooth submanifold (with boundary) of $\bR^n$ by Proposition 5.49(b) in \cite{lee2012introduction}.
\end{proof}

The following result implies that $(\cM,\langle\mdot,\mdot\rangle^\cM)$ is a Riemannian submanifold of the function space $\bR^n$ of predictors.
\begin{lemma}
    For any $w \in B$, let
    \begin{equation}
        \cT_{\alpha f(w)}\cM := \{\alpha \rr{D}f(w)z: z\in\bR^p\} = \rr{Im}(\rr{D}f(w)).
    \end{equation}
    Then, $\cT_{\alpha f(w)}\cM$ is the tangent space of $\alpha f(w)\in\cM$. Also, for any $w\in B$ and $u,v\in\cT_{\alpha f(w)}\cM$, $\langle u,v\rangle^\cM_{\alpha f(w)}:= \langle u, v\rangle=u^\top v$ is a Riemannian metric on $\cM$. Consequently, $\left(\cM,\langle\cdot,\cdot\rangle^\cM\right)$ is a Riemannian submanifold of $\bR^n$.
\end{lemma}
\begin{proof}
    Note that the tangent space for $\cM=\alpha f(B)$ is defined as $$\cT_{\alpha f(w)}\cM:=\{c'(0): c:I\rightarrow\cM\mbox{ is smooth},~c(0)=\alpha f(w)\},$$ where $I\subset \bR$ is any interval with $0\in I$ \citep{absil2010optimization, lee2018introduction, boumal2023introduction}. Let $I = (-\epsilon,\epsilon)$ for $\epsilon > 0$. Since $f:\bR^p\rightarrow\bR^n$ is a smooth embedding, if $c:I\rightarrow\cM$ is a smooth curve on $\cM=\alpha f(B)$, then there exists a smooth curve $\gamma:I\rightarrow B$ such that $c(t) = f(\gamma(t))$ for $t\in I$, with $\gamma(0)=w$. Then, we have $$\dt{c(t)}=\dt{\alpha f(\gamma(t))} = \rr{D}f(\gamma(t))\dt{\gamma(t)},$$ by the chain rule. Thus, $c'(0)=\rr{D}f(w)\gamma'(0)\in\rr{Im}(\rr{D}f(w))$. The second part of the claim is a direct consequence of Proposition \ref{prop:submanifold}, as the restriction of the Euclidean metric to an embedded submanifold of $\bR^n$ ($\cM$ in our case, by Proposition \ref{prop:submanifold}) is a Riemannian metric \citep{boumal2023introduction}.
\end{proof}

The following result shows that the Gauss-Newton dynamics in the underparameterized regime corresponds to a Riemannian gradient flow in the function space.
\begin{proposition}[Gauss-Newton as a Riemannian gradient flow]
    For any $\alpha f(w)\in\cM$, $\bm{P}(\alpha f(w))$ is the projection operator onto its tangent space $\cT_{\alpha f(w)}\cM$, i.e., $$\bm{P}(\alpha f(w))z = \underset{y\in\cT_{\alpha f(w)}\cM}{\arg\min}~\|y-z\|^2.$$ Furthermore,
    \begin{equation}
        \rr{grad}_{\alpha f(w)}^\cM g(\alpha f(w)) := \bm{P}(\alpha f(w))\nabla g(\alpha f(w))
    \end{equation}
    is the Riemannian gradient of $g$ at $\alpha f(w)\in\cM$. Consequently, the Gauss-Newton dynamics in \eqref{eqn:gauss-newton}, i.e., $$\dt{\alpha f(w_t)}=-\bm{P}(\alpha f(w_t))\nabla g(\alpha f(w_t)) = \rr{grad}_{\alpha f(w_t)}^\cM g(\alpha f(w_t)),$$ corresponds to Riemannian gradient flow on $(\cM,\langle\cdot,\cdot\rangle^\cM)$.
    \label{prop:riemannian-gf}
\end{proposition}
\begin{proof}[of Proposition \ref{prop:riemannian-gf}]
    Since $\rank(\rr{D}f(w))=p$ for any $w\in B$, $\bm{P}(\alpha f(w))$ is well-defined on $B$. First, notice that $\bm{P}^\top(\alpha f(w))=\bm{P}(\alpha f(w))$ and $\bm{P}^2(\alpha f(w))=\bm{P}(\alpha f(w))$ (i.e., $\bm{P}(\alpha f(w))$ is idempotent), thus $\bm{P}(\alpha f(w))$ is a projection matrix onto a $p$-dimensional subspace of $\bR^n$. Since $\cT_{\alpha f(w)}\cM=\rr{Im}(\rr{D}f(w))$, let $$\pi_{\cT_{\alpha f(w)}\cM}[z]:=\argmin{u\in\cT_{\alpha f(w)}\cM}\|u-z\|_2^2=\argmin{v\in\bR^p}\|z-\rr{D}f(w)v\|_2^2.$$ By using first-order condition for global optimality, we have $2\rr{D}^\top f(w)(\rr{D}f(w)v^\star-z)=0$, which implies that $\rr{D}f(w)v^\star = \bm{P}(\alpha f(w))z \in \pi_{\cT_{\alpha f(w)}\cM}[z]$ is the unique minimizer. As such, $$\grad_{\alpha f(w_t)}^\cM g(\alpha f(w_t)) = \pi_{\cT_{\alpha f(w_t)}\cM}[\nabla g(\alpha f(w_t))],$$ thus it is the Riemannian gradient of $g(\alpha f(w_t))$ by Prop. 3.61 in \cite{boumal2023introduction}.
\end{proof}

In Figure \ref{fig:manifold}, we illustrate the training trajectories of a single-neuron (i.e., $m=1$) with $\tanh$ activation function in the function and parameter spaces on a problem with $n = 3$ random data points of dimension $d=2$. The embedded submanifold $\cM=\alpha f(B)$ is the two-dimensional surface in the function space $\bR^3$, as illustrated in Figure \ref{fig:func-space}.

\begin{figure}[!ht]
    \centering
    \begin{subfigure}[t]{0.5\textwidth}
        \centering
        \includegraphics[width=.9\linewidth]{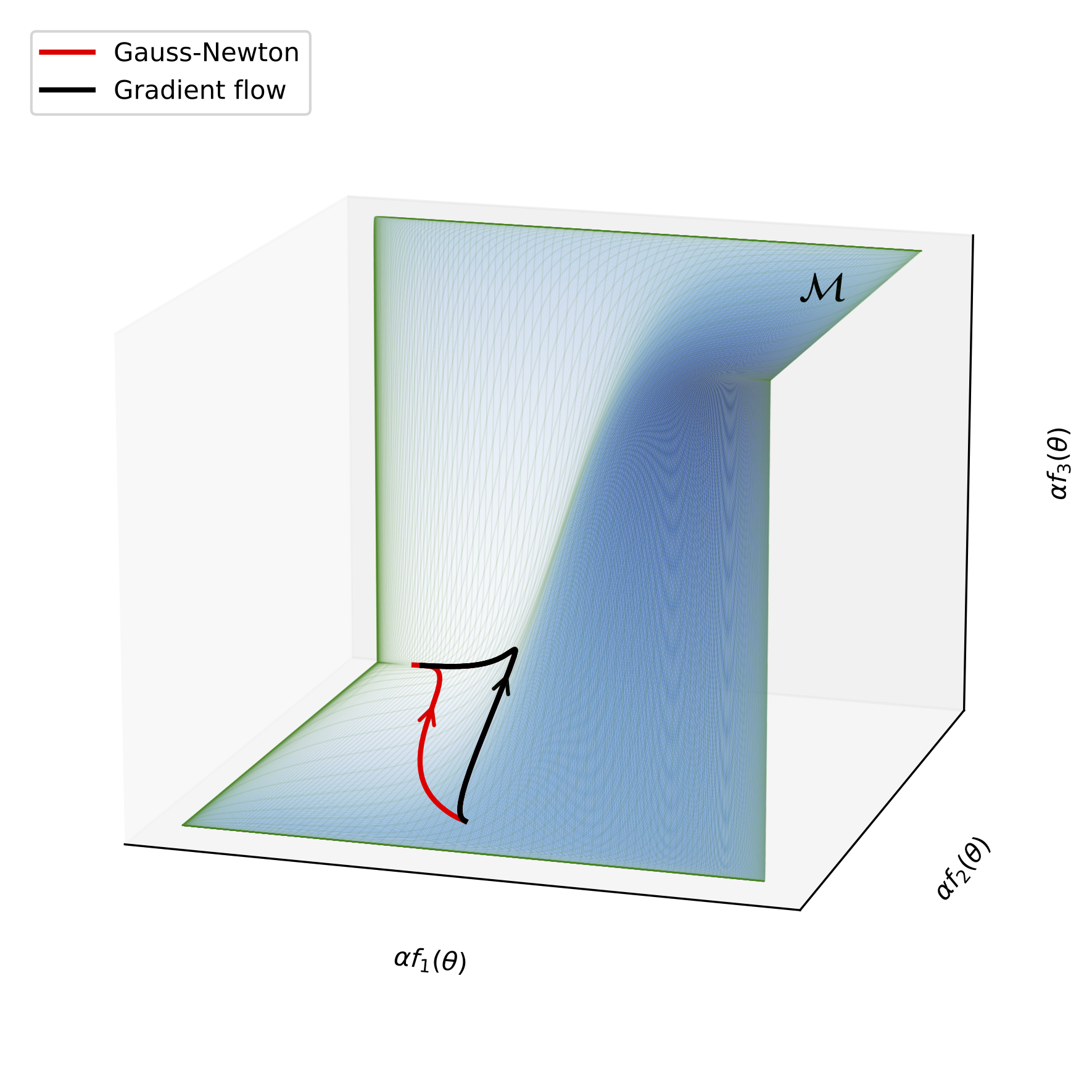}
        \caption{Evolution of the predictor on the two-dimensional Riemannian submanifold $\cM$ of the function space $\bR^3$.}
        \label{fig:func-space}
    \end{subfigure}%
    ~ 
    \begin{subfigure}[t]{0.5\textwidth}
        \centering
        \includegraphics[width=.9\linewidth]{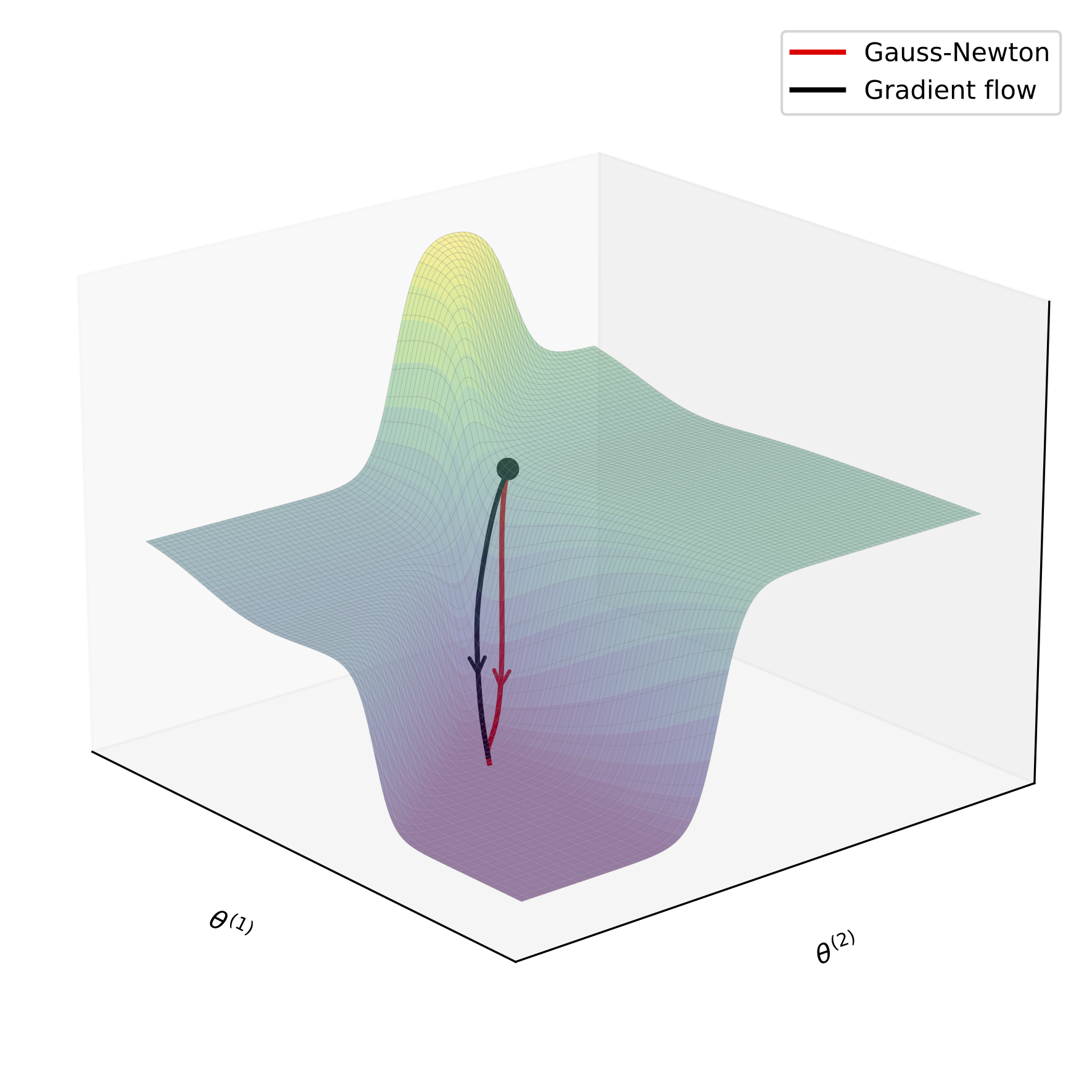}
        \caption{Evolution in the parameter space.}
        \label{fig:par-space}
    \end{subfigure}
    \caption{Trajectories of the Gauss-Newton and the gradient flow in the function space and the parameter space for $n=3$ and $p=2$. Gauss-Newton induces Riemannian gradient flow on $\cM$.}
    \label{fig:manifold}
\end{figure}

As a consequence of Prop. \ref{prop:riemannian-gf}, we will utilize the tools from optimization on smooth Riemannian manifolds to analyze the convergence and optimality of the predictor under the Gauss-Newton dynamics.

\subsection{Convergence of the Gauss-Newton Dynamics in the Underparameterized Regime}
In this section, we will establish various geodesic convexity and Lipschitz-continuity results on $\cM$, which will lead us to the convergence bounds for \eqref{eqn:ei-u}. 

We first prove the geodesic convexity of $\cM$, which will play a fundamental role in the convergence proof.

\begin{lemma}\label{lemma:g-convexity-M}
    $\cM$ is a geodesically convex set. 
\end{lemma}
\noindent Lemma \ref{lemma:g-convexity-M} implies that for any points $\alpha f(w), \alpha f(w')\in\cM$, there exists a geodesic segment $c:[0,1]\rightarrow\cM$ such that $c(0)=\alpha f(w)$ and $c(1)=\alpha f(w')$. Since $\alpha f|_B$ is generally non-convex and $\cM$ is a curved manifold, this result does not automatically hold. The existence of such a geodesic segment is guaranteed when $\sup_{w,w'\in B}\|\alpha f(w)-\alpha f(w')\|\leq\frac{\pi}{\Lambda}$, where $$\Lambda = \sup_{\alpha f(w)\in\cM}\sup_{u:\|u\|_{\alpha f(w)}=1}\|II_{\alpha f(w)}(u,\cdot)\|_\rr{op},$$ and $II_p$ is the second fundamental form \citep{alexander1993geometric, BridsonHaefliger1999}. To establish this condition, we prove a uniform upper bound on $\|II_{\alpha f(w)}(u,\cdot)\|_\rr{op}$ via a variational formulation in Lemma \ref{lemma:curvature}, and use metric compactness of $\cM$ to conclude the result. The complete proof can be found in Appendix \ref{app:up}.


In the following, we establish key geodesic regularity conditions for the loss function on the manifold $\cM$ from first principles, which constitutes an essential part of the convergence analysis. To that end, let $$\cS := \{y\in\cM:g(y) \leq g(0)\},$$ which is a nonempty set since $g(\alpha f(w_0))=0$, thus $\alpha f(w_0)\in\cS$. As a consequence of the energy dissipation inequality \eqref{eqn:evi-u}, $t\mapsto g(\alpha f(w_t))$ is a non-increasing function on $t<T$, thus under Gauss-Newton dynamics, the optimization trajectory lies in $\cS$: $$\{\alpha f(w_t):t\in[0,T)\}\subset \cS \subset \cM.$$

\begin{theorem}[Geodesic strong convexity and smoothness of $g$]
    For any \begin{equation}\label{eqn:alpha0}\alpha \geq \frac{4L\mu}{\nu\lambda_0^2}\max\left\{\|f^\star\|, \sqrt{\frac{2g(0)}{\nu}}\right\}=:\alpha_0,\end{equation} the following are true.
    \begin{enumerate}[label=(\alph*)]
        \item The loss function $g|_\cM:\cM\rightarrow\bR$ is geodesically convex, i.e., for any $z\in\cM, v\in\cT_{z}\cM$ and $c(t)=\rr{Exp}_{z}(vt),~t\in[0,1]$, we have $$g(z) + t\langle \grad_{z}^\cM g(z), v\rangle_{z}^\cM \leq g(\rr{Exp}_{z}(tv)),$$
    for any $t\in[0,1]$.
        \item The sublevel set $$\cS := \{z\in\cM: g(z)\leq g(0)\}$$ is geodesically convex.
        \item $g|_\cS$ is a $\frac{\nu}{2}$-geodesically strongly convex function on $\cS$: for any $z\in\cS, v\in\cT_{z}\cM$ and $c(t)=\rr{Exp}_{z}(vt)$ for $t\in[0,1]$, we have $$g(z) + t\langle \grad_{z}^\cM g(z), v\rangle_{z}^\cM \leq g(\rr{Exp}_{z}(tv))-t^2\frac{\nu}{4}\|v\|^2,$$
    for any $t\in[0,1]$.
    \item $g|_\cS$ has geodesically $\frac{3\mu}{2}$-Lipschitz continuous gradients: for any $z\in\cS, v\in\cT_{z}\cM$ and $c(t)=\rr{Exp}_{z}(vt)$ for $t\in[0,1]$, we have $$g(z) + t\langle \grad_{z}^\cM g(z), v\rangle_{z}^\cM \geq g(\rr{Exp}_{z}(tv))-t^2\frac{3\mu}{4}\|v\|_{z}^2,$$
    for any $t\in[0,1]$.
    \end{enumerate}
    \label{thm:geodesic-regularity}
\end{theorem}
\noindent The proof of Theorem \ref{thm:geodesic-regularity}, which is provided in Appendix \ref{sec:up}, relies on establishing lower bounds on the minimum singular value of the Riemannian Hessian of $g$ on $\cM$ and $\cS$.

\begin{remark}[Controlling the curvature by $\alpha$]
 A key step in the proof of Theorem \ref{thm:geodesic-regularity} is to bound $\big\|\dt{}\big[\bm{P}(\alpha f(\gamma_t)\big]\big|_{t=0}\big\|$ where $\gamma_t:[0,1]\rightarrow B$ is any smooth curve such that $\gamma_0 = w\in B$ and $\dt{\alpha f(\gamma_t)}|_{t=0}=u\in\cT_{\alpha f(w)}\cM\backslash\{0\}$, which corresponds to the magnitude of the second fundamental form along the direction $u$, and thus quantifies curvature. In order to establish the geodesic regularity of $g$, we choose $\alpha$ sufficiently large to control the curvature of $\cM$.
\label{remark:curvature}
\end{remark}

Since $\cM$ is a compact and smooth embedded Riemannian manifold, there exists an in-class optimal predictor $f^\star_\cM$ such that 
$$
g(f^\star_\cM)=\inf\{g(z):z\in\cM\}.
$$
We make the following representational assumption regarding the existence of a critical point in the interior $\cM$.
\begin{assumption}\label{assumption:interior}
    There exists $w_\alpha^\star \in \rr{int}(B)$ such that $\grad_{\alpha f(w_\alpha ^\star)}^\cM~g(\alpha f(w_\alpha^\star)) = 0$.
\end{assumption}
\noindent By the first-order condition for optimality for geodesically convex functions (Prop. 4.6 and Corollary 11.22 in \cite{boumal2023introduction}), Assumption \ref{assumption:interior} just states that the optimal predictor in $\cM$ is in the interior.

Recall from the analyses in Section \ref{sec:op} (e.g., Theorem \ref{thm:lm-convergence}) that the convergence analysis relies on the following consequences of Euclidean strong convexity and Lipschitz smoothness:
\begin{enumerate}[label=(\roman*)]
    \item $\|\nabla g(\alpha f(w))\|\leq \mu\|\alpha f(w)-f^\star\|\leq \mu\sqrt{\frac{2}{\nu}(g(\alpha f(w))-g(f^\star))}$,
    \item the Polyak-\L ojasiewicz inequality $\|\nabla g(\alpha f(w))\|^2 \geq 2\nu (g(\alpha f(w))-g(f^\star))$.
\end{enumerate}
By using the geodesic regularity conditions established in Theorem \ref{thm:geodesic-regularity}, the following lemma establishes Riemannian analogues of the above results on the manifold $\cM$.

\begin{lemma}\label{lemma:riemannian-pl}
    For any sufficiently large $\alpha\geq \alpha_0$ where $\alpha_0$ is given in \eqref{eqn:alpha0}, for all $t\in[0,T)$, there exists a tangent vector $v_t\in\cT_{\alpha f(w^\star)}\cM$ such that
    \begin{align}
        \label{eqn:pl-u}&\frac{\nu}{4}\|v_t\|_{2}^2 \leq V_t \leq \frac{1}{\nu}\|\grad_{\alpha f(w_t)}^\cM g(\alpha f(w_t))\|_2^2\\
        &\|\grad_{\alpha f(w_t)}^\cM g(\alpha f(w_t))\|_2\leq \frac{3\mu}{2}\|v_t\|_2,
    \end{align}
    where $$V_t:=g(\alpha f(w_t))-g(\alpha f(w_\alpha^\star))=g(\alpha f(w_t))-\inf_{z\in\alpha f(B)}g(z),\quad t\geq 0.$$
\end{lemma}

In the following, we present the main convergence and optimality result for the Gauss-Newton gradient flow in the underparameterized regime, which is the main result in this section.

\begin{theorem}[Convergence in the underparameterized regime]
    For the scaling factor $$\alpha = \max\left\{\alpha_0, \frac{3\mu\sqrt{g(0)}}{\lambda_0\nu^\frac{3}{2}r_0}\right\},$$ the Gauss-Newton dynamics achieves
    \begin{equation}
        V_t \leq g(0)\exp(-\nu t)\mbox{ for any }t \in [0,\infty),
    \end{equation}
    under Assumptions \ref{assumption:gram}-\ref{assumption:interior}, where $V_t := g(\alpha f(w_t))-g(\alpha f(w_\alpha^\star))$. Furthermore, in the same setting, \begin{equation}\|w_t-w_0\|_2\leq r_0\mbox{ and }\rr{D}^\top f(w_t)\rr{D}f(w_t)\succcurlyeq \lambda_0^2\bm{I},
    \label{eqn:deviation}
    \end{equation} 
    for any $t\in\bR^+$.
    \label{thm:gf-convergence}
\end{theorem}

\begin{proof}
    From Lemma \ref{lemma:si-evi-u}, recall that we have $$\dt{V_t} = -\|\grad_{\alpha f(w_t)}^\cM g(\alpha f(w_t))\|_2^2~\mbox{for any }~t\in(0,T).$$
    For $\alpha \geq \alpha_0$, $g|_\cS$ is $\nu/2$ geodesically strong convex by Theorem \ref{thm:geodesic-regularity}. Hence, the Riemannian Polyak-\L ojasiewicz condition in Lemma \ref{lemma:riemannian-pl} implies that $$\dot{V}_t\leq -\nu V_t.$$ Thus, Grönwall's lemma implies
    \begin{equation}
        V_t\leq V_0\exp(-\nu t)\mbox{ for any }t\in[0,T).
        \label{eqn:error-bound-u}
    \end{equation}
    To show that $T=\infty$, take $t \in [0,T)$. Then,
    \begin{align}
        \nonumber \|w_t-w_0\|_2 &\leq \int_0^t\|\dot{w}_s\|\rr{d}s\\
        &= \frac{1}{\alpha}\int_0^t \|\nabla g(\alpha f(w_t))\|_{\bm{A}_s}\rr{d}s,\label{eqn:dev-bound}
    \end{align}
    where $$\bm{A}_s := \rr{D}f(w_s)\bm{H}_0^{-2}(\alpha f(w_s))\rr{D}^\top f(w_s) \mbox{ for }s<T.$$
Since $s < t < T$, we have $w_s\in B$, thus $\bm{H}_0(\alpha f(w_s))\succcurlyeq \lambda_0^2\bm{I}$. This implies that
\begin{align}
    \nonumber \|\nabla g(\alpha f(w_s))\|_{\bm{A}_s}^2 &\leq \frac{1}{\lambda_0^2}\|\nabla g(\alpha f(w_s))\|_{\bm{P}(\alpha f(w_s))}^2\\
     \nonumber &= \frac{1}{\lambda_0^2}\|\bm{P}(\alpha f(w_s))\nabla g(\alpha f(w_s))\|_2^2\\
     &= \frac{1}{\lambda_0^2}\|\grad_{\alpha f(w_s)}^\cM g(\alpha f(w_s))\|_{2}^2.\label{eqn:grad-bound-i}
\end{align}
By using \eqref{eqn:pl-u} in Lemma \ref{lemma:riemannian-pl}, we have 
\begin{align*}
    \|\grad_{\alpha f(w_s)}^\cM g(\alpha f(w_s))\|_{\alpha f(w_s)}^2 \leq \frac{9\mu^2}{4}\|v_s\|_2^2\leq \frac{9\mu^2}{\nu}V_s.
\end{align*}
Using the error bound \eqref{eqn:error-bound-u}, we obtain
\begin{equation}
    \|\grad_{\alpha f(w_s)}^\cM g(\alpha f(w_s))\|_{\alpha f(w_s)}\leq \frac{3\mu\sqrt{V_s}}{\sqrt{\nu}}\leq \frac{3\mu\sqrt{V_0}}{\sqrt{\nu}}e^{-\nu s} \leq \frac{3\mu\sqrt{g(0)}}{\sqrt{\nu}}e^{-\nu s}\mbox{ for any }s < T.
    \label{eqn:grad-bound-ii}
\end{equation}
Substituting \eqref{eqn:grad-bound-i} and \eqref{eqn:grad-bound-ii} into \eqref{eqn:dev-bound}, we obtain 
\begin{align}
    \nonumber \|w_t-w_0\| &\leq \frac{3\mu\sqrt{g(0)}}{\alpha\lambda_0\sqrt{\nu}}\int_0^t\exp(-\nu s)\rr{d}s\leq \frac{3\mu\sqrt{g(0)}}{\alpha\lambda_0\sqrt{\nu^3}}.
\end{align}
Hence, $\alpha \geq \frac{3\mu\sqrt{g(0)}}{\lambda_0r_0\nu^\frac{3}{2}}$ yields $\|w_t-w_0\|\leq r_0$ and $\rr{D}^\top f(w_t)\rr{D}f(w_t) \succcurlyeq \lambda_0^2\bm{I}$. Hence, $T=\infty$.
\end{proof}

\begin{remark}[Benefits of preconditioning in the underparameterized regime] \normalfont
We have the following observations on the superiority of the Gauss-Newton gradient flow in the underparameterized regime compared to the gradient flow.
\begin{itemize}
    \item \textbf{Exponential convergence rate for the last-iterate.} The Gauss-Newton gradient flow achieves \emph{exponential} convergence rate $\exp(-\Omega(t))$ for the last-iterate in the underparameterized regime. The convergence rate for gradient descent in this regime is subexponential \citep{ji2020polylogarithmic, cayci2024convergence} under compatible assumptions.

    \item \textbf{Convergence without explicit regularization.} The convergence result in Theorem \ref{thm:gf-convergence} holds \emph{without} any explicit regularization scheme, e.g., early stopping or projection. The Gauss-Newton gradient flow converges self-regularizes in the underparameterized setting as in \eqref{eqn:deviation}. In the underparameterized regime, the gradient descent dynamics requires an explicit regularization scheme to control the parameter movement $\|w_t-w_0\|$, e.g., early stopping \citep{ji2020polylogarithmic, cayci2024convergence} or projection \citep{cai2019neural, cayci2023sample, cayci2024convergence}.
    \end{itemize}
    The Riemannian gradient flow interpretation of the Gauss-Newton dynamics is key in establishing the above results. We should note that we do not follow the analysis based on projected subgradient descent as in the analyses of gradient descent in the underparameterized regime, which leads to these fundamental differences \citep{ji2020polylogarithmic, cayci2024convergence}.
    \begin{itemize}
    \item \textbf{$\lambda_0$-independent convergence rate.} The convergence rate in Theorem \ref{thm:gf-convergence} is independent of the minimum eigenvalue $\lambda_0$ of the Gram matrix $\rr{D}^\top f\rr{D}f$, which indicates that the performance of the Gauss-Newton dynamics is resilient against ill-conditioned Gram matrices due to the geometry of the input data points $\{x_j\}_{j=1,\ldots,n}\subset \bR^d$.
    
\end{itemize}
    
\end{remark}



\begin{remark}[Regularization by the scaling factor $\alpha$]\normalfont\label{remark:alpha}
    Theorem \ref{thm:gf-convergence} implies that the scaling factor $\alpha$ controls $\|w_t-w_0\|_2$. Equation \eqref{eqn:deviation} indicates that 
    \begin{itemize}
        \item $\alpha > 0$ should be large enough to ensure that $\rr{D}^\top f(w_t)\rr{D}f(w_t),~t\in\bR^+$ is strictly positive-definite,
        \item $\alpha \uparrow \infty$ leads to a smaller parameter set over which $g\circ (\alpha f)$ is optimized, which leads to increasing inductive bias $\inf_{y\in\bR^n}g(f)-\inf_{w\in B}g(\alpha f(w))$.
    \end{itemize}
    This phenomenon is unique to the underparameterized setting (since the optimality gap is defined with respect to the best in-class predictor unlike the overparameterized case), and will be illustrated in the numerical example in Appendix \ref{sec:num}.
\end{remark}

\section{Conclusions}
In this work, we analyzed the Gauss-Newton dynamics for underparameterized and overparameterized neural networks in the near-initialization regime, and demonstrated that the recent optimization tools developed for embedded submanifolds can provide important insights into the training dynamics of neural networks. As a follow-up to this work, the performance analysis of the Gauss-Newton method in mean-field regime \citep{chizat2018global, mei2019mean, sirignano2020mean}, is an interesting future direction.



\newpage

\appendix
\section{Omitted Proofs in Section \ref{sec:op}}\label{app:op}

\begin{proof}[of Lemma \ref{lemma:si-evi}]
    By the chain rule, we obtain
    \begin{align*}
        \dt{\alpha f(w_t)}&=-\alpha\rr{D}f(w_t)\dt{w_t}\\
        &=-\rr{D}f(w_t)[\Ht]^{-1}\rr{D}^\top f(w_t)\nabla g(\alpha f(w_t)).
    \end{align*}
    Since $t < T$, the empirical kernel matrix $\bm{K}_t$ is non-singular. Thus, by applying the Sherman-Morrison-Woodbury matrix identity \citep{horn2012matrix} to the above, we obtain
    \begin{align*}
        \rr{D}f(w_t)&[\Ht]^{-1}\rr{D}^\top f(w_t) \\&= \frac{1}{\rho_t}\rr{D}f(w_t)\left[\bm{I}-\frac{1-\rho_t}{\rho_t}\rr{D}f(w_t)\left[\bm{I}+\frac{1-\rho_t}{\rho_t}\rr{D}f(w_t)\rr{D}^\top f(w_t)\right]^{-1}\rr{D}f(w_t)\right]\rr{D}^\top f(w_t)\\
        &= \frac{1}{\rho_t}\Big(\bm{K}_t - \frac{1-\rho_t}{\rho_t}\bm{K}_t\Big(\bm{I}+\frac{1-\rho_t}{\rho_t}\bm{K}_t\Big)^{-1}\bm{K}_t\Big).
    \end{align*}
    This gives the gradient flow in the output space \eqref{eqn:ei}.
    
    For the energy dissipation inequality \eqref{eqn:evi}, first note that we have
    \begin{align}
       \nonumber \dt{g(\alpha f(w_t))} &= \dt{V_t}\\
        \nonumber &= \nabla^\top g(\alpha f(w_t)) \dt{\alpha f(w_t)}\\
        &= - \|\nabla g(\alpha f(w_t))\|_{\rr{D}f(w_t)[\Ht]^{-1}\rr{D}^\top f(w_t)}^2\label{eqn:evi-identity}
    \end{align}
    where the last identity comes from \eqref{eqn:ei}. As such, we need to characterize the spectrum, particularly the minimum singular value of $\bm{K}_t - \frac{1-\rho_t}{\rho_t}\bm{K}_t\big(\bm{I}+\frac{1-\rho_t}{\rho_t}\bm{K}_t\big)^{-1}\bm{K}_t$. To that end, let $(\gamma,u)\in\bR\times\bR^n$ be any eigenvalue-eigenvector pair for the matrix $\bm{K}_t$. Then,
    \begin{align*}
        \big(\bm{K}_t - \frac{1-\rho_t}{\rho_t}\bm{K}_t\big(\bm{I}+\frac{1-\rho_t}{\rho_t}\bm{K}_t\big)^{-1}\bm{K}_t\big)u &= \gamma u - \frac{1-\rho_t}{\rho_t}\frac{\gamma^2}{1+\frac{1-\rho_t}{\rho_t}\gamma}u\\
        &= \frac{\gamma\frac{\rho_t}{1-\rho_t}}{\frac{\rho_t}{1-\rho_t}+\gamma}u,
    \end{align*}
    which implies that $(\frac{\gamma\rho_t}{(1-\rho_t)\gamma+\rho_t}, u)$ is an eigenpair for $\bm{K}_t - \frac{1-\rho_t}{\rho_t}\bm{K}_t\big(\bm{I}+\frac{1-\rho_t}{\rho_t}\bm{K}_t\big)^{-1}\bm{K}_t$, and therefore $\rr{D}f(w_t)[\Ht]^{-1}\rr{D}^\top f(w_t)$ has a corresponding eigenpair $\left(\frac{\gamma}{(1-\rho_t)\gamma+\rho_t},u\right)$. Then, for any $(\rho_t)_{t\geq 0}$ with $\inf_{t\geq 0}\rho_t>0$:
    \begin{equation}
        \|\nabla g(\alpha f(w_t))\|_{\rr{D}f(w_t)[\Ht]^{-1}\rr{D}^\top f(w_t)}^2 \geq \|\nabla g(\alpha f(w_t))\|_2^2\cdot \frac{\lambda_t^2}{(1-\rho_t)\lambda_t^2+\rho_t}.
        \label{eqn:grad-norm-gn}
    \end{equation}
    Substituting \eqref{eqn:grad-norm-gn} into \eqref{eqn:evi-identity} concludes the proof of Lemma \ref{lemma:si-evi}.
    \end{proof}

\begin{proof}[of Lemma \ref{lemma:edi}]
        Note that $\dt{V_t}=\dt{g(\alpha f(w_t))}$ and $\lambda_t^2\geq\lambda^2$ for any $t < T$ by \eqref{eqn:ntk-bound}, thus \eqref{eqn:evi} with constant $\rho>0$ implies
    \begin{align}
        \label{eqn:opt-gap-evol-const} \dt{V_t} &\leq -\frac{\lambda^2}{(1-\rho)\lambda^2+\rho}\|\nabla g(\alpha f(w_t))\|_2^2
    \end{align}
    since $z\mapsto \frac{z}{(1-\rho)z+\rho}$ is a monotonically increasing function for $\rho \geq 0$ and $\lambda_t^2\geq \lambda^2$. Since $f\mapsto g(f)$ is $\nu$-strongly convex, by Polyak-\L{}ojasiewicz (P\L) inequality \cite{bach2024learning}, we have 
    $$\|\nabla g(\alpha f(w_t))\|_2^2 \geq 2\nu V_t.$$ 
    Substituting this outcome of the P\L{}-inequality and \eqref{eqn:grad-norm-gn} into \eqref{eqn:evi-identity}, we obtain
    $$\dt{V_t}\leq -\frac{\lambda^2}{(1-\rho)\lambda^2+\rho}\cdot 2\nu V_t,\quad t\in[0,T).$$
    Thus, by Gr\"{o}nwall's lemma \citep{terrell2009stability}, we obtain
    \begin{align}\label{eqn:gap-const}
        V_t &\leq V_0\exp\left(-\frac{2\nu\lambda^2t}{(1-\rho)\lambda^2+\rho}\right) 
    \end{align}
    for any $t \in [0,T)$. Now, note that 
    \begin{equation}\label{eqn:forward-ineq}
    g(f^\star) \leq g(\alpha f(w_t)) - \frac{\nu}{2}\|\alpha f(w_t)-f^\star\|_2^2,
    \end{equation} since $f^\star \in \underset{x\in\bR^n}{\arg\min}~g(x)$ is the unique minimizer of the strongly convex $g:\bR^n\rightarrow\bR^+$, which implies that $\nabla g(f^\star) = 0$ by the first-order condition for optimality \citep{boyd2004convex}. Thus,
    \begin{equation}
        \|\alpha f(w_t)-f^\star\|_2^2 \leq \frac{2V_t}{\nu}
        \label{eqn:gap-ii}
    \end{equation}
    for any $t < T$. Substituting \eqref{eqn:gap-const} into \eqref{eqn:gap-ii} concludes the proof.
    \end{proof}

\begin{proof}[of Lemma \ref{lemma:non-degeneracy}] For a constant damping scheme with $\rho_t=\rho\in(0,\lambda^2/(1+\lambda^2)]$, by the triangle inequality, we have 
    \begin{align}
        \nonumber \|w_t-w_0\|_2 = \left\|\int_0^t\dot{w}_sds\right\|_2 &\leq \int_{0}^t\|\dot{w}_s\|_2\rr{d}s\\
        &=\frac{1}{\alpha}\int_{0}^t \|\nabla g(\alpha f(w_s))\|_{\rr{D}f(w_s)[\bm{H}_\rho(\alpha f(w_s)]^{-2}\rr{D}^\top f(w_s)}\rr{d}s.\label{eqn:deviation-bound}
    \end{align}
    Using the Sherman-Morrison-Woodbury matrix identity \cite{horn2012matrix}, we obtain
    \begin{multline}
        (1-\rho)^2\rr{D}f(w_t)[\bm{H}_\rho(\alpha f(w_t)]^{-2}\rr{D}^\top f(w_t) \\ = \rho_0^{-2}\bm{K}_t-2\rho_0^{-3}\bm{K}_t(\bm{I}+\rho_0^{-1}\bm{K}_t)^{-1}\bm{K}_t + \rho_0^{-4}\bm{K}_t(\bm{I}+\rho_0^{-1}\bm{K}_t)^{-1}\bm{K}_t(\bm{I}+\rho_0^{-1}\bm{K}_t)^{-1}\bm{K}_t,
    \end{multline}
    where $\rho_0:=\frac{\rho}{1-\rho}$. For any $z_1 \geq z_0 \geq \rho_0$, we have 
    \begin{equation}
        \frac{z_1}{((1-\rho)z_1+\rho)^2}\leq \frac{z_0}{((1-\rho)z_0+\rho)^2}.
        \label{eqn:max-eig}
    \end{equation} Then, if $(\gamma,u)\in\bR\times\bR^n$ is an eigenpair for $\bm{K}_t$, then $\Big(\frac{\gamma}{((1-\rho)\gamma+\rho)^2}, u\Big)$ is an eigenpair for the positive-definite matrix $\rr{D}f(w_t)[\bm{H}_\rho(\alpha f(w_t))]^{-2}\rr{D}^\top f(w_t)$. Furthermore, by \eqref{eqn:max-eig} and \eqref{eqn:ntk-bound}, the maximum eigenvalue of $\rr{D}f(w_t)[\bm{H}_\rho(\alpha f(w_t))]^{-2}\rr{D}^\top f(w_t)$ is upper bounded by $\frac{\lambda^2}{\left((1-\rho)\lambda^2+\rho\right)^2}$. Thus, we have
    \begin{align}
    \begin{aligned}
        \|\nabla g(\alpha f(w_s))\|_{\rr{D}f(w_s)[\bm{H}_\rho(\alpha f(w_s)]^{-2}\rr{D}^\top f(w_s)}^2 \leq \frac{\lambda^2}{((1-\rho)\lambda^2+\rho)^2}\|\nabla g(\alpha f(w_s))\|_2^2,
        \end{aligned}
        \label{eqn:op-norm-bound}
    \end{align}
    for $s \leq t < T$, which is implied by $\rho \leq \frac{\lambda^2}{1+\lambda^2}$. Since $s \leq t < T$, we have
    \begin{align}
        \nonumber \|\nabla g(\alpha(f(w_s))\|_2 &= \|\nabla g(\alpha(f(w_s))-\nabla g(f^\star)\|_2\\
        \label{eqn:gap-grad-bound}&\leq \mu\|\alpha f(w_s)-f^\star\|\\
        \nonumber &\leq \mu\sqrt{\frac{2V_0}{\nu}}\exp\Big(-\frac{\nu\lambda^2s}{\rho+(1-\rho)\lambda^2}\Big),
    \end{align}
    where the first line follows from the global optimality of $f^\star$, the second line is due to $\mu$-Lipschitz-continuous gradients of $g$, and the last line follows from the bound on the optimality gap in \eqref{eqn:gap-ii}. Thus,
    \begin{align*}
        \|\nabla g(\alpha f(w_s))\|_{\rr{D}f(w_s)[\bm{H}_\rho(\alpha f(w_s)]^{-2}\rr{D}^\top f(w_s)} &\leq \frac{\mu\sqrt{\frac{2V_0}{\nu}}\lambda}{(1-\rho)\lambda^2+\rho}\exp\Big(-\frac{\nu\lambda^2s}{\rho+(1-\rho)\lambda^2}\Big)\mbox{ for any }s<T.
    \end{align*}
    Substituting the above inequality into \eqref{eqn:deviation-bound}, we obtain
    \begin{align*}
        \|w_t-w_0\|_2 &\leq \frac{1}{\alpha}\cdot \frac{\mu\sqrt{2V_0}}{\nu^{3/2}}\cdot \frac{1}{\lambda}\leq r_0 := \frac{\lambda}{L}
    \end{align*}
    with the choice of $\alpha \geq \frac{\mu L\sqrt{2V_0}}{\nu^{3/2}}\cdot \frac{1}{\lambda^2}$, where $L$ is given in \eqref{eqn:L}. Therefore, $\sup_{t<T}\|w_t-w_0\|_2\leq r_0$, where $r_0$ is independent of $T$. We conclude that $T=\infty$.
\end{proof}

\begin{proof}[of Theorem \ref{thm:lm-convergence-dt}] The proof heavily relies on the continuous time analysis of Theorem \ref{thm:lm-convergence} and the discretization idea in \cite{du2018gradient}. For notational simplicity, let $\bm{D}_k:= \rr{D}f(w_k)$, $\bm{H}_k := \bm{H}_\rho(f(w_k))$ for $k\in\bN$.

Recall that $w\mapsto \rr{D}f(w)$ is $L$-Lipschitz with $$L=\frac{\sigma_2}{\sqrt{m}}\sqrt{\sum_{j=1}^n\|x_j\|^4}\leq \sigma_2\sqrt{\frac{n}{m}},$$ under the assumption that $\max_{1\leq j\leq n}\|x_j\|\leq 1$. Let $N:=\inf\{k\in\bN:\|w_k-w_0\|_2>\frac{\lambda_0}{L}\}$ and consider $k<N$. Since $g$ has $\mu$-Lipschitz gradients, we have
    \begin{equation}
        g(f(w_{k+1}))\leq g(f(w_{k})) + \nabla ^\top g(f(w_{k}))[f(w_{k+1})-f(w_k)] + \frac{\mu}{2}\|f(w_{k+1})-f(w_k)\|_2^2.
        \label{eqn:lyapunov}
    \end{equation}
    Since $w\mapsto f(w)$ has $L$-Lipschitz gradients, we have $$f(w_{k+1})-f(w_k)=\bm{D}_k(w_{k+1}-w_k)+\bm{\epsilon}_k,$$ where the local linearization error is bounded as $$\|\bm{\epsilon}_k\|_2\leq \frac{L}{2}\|w_{k+1}-w_k\|_2^2.$$ Define
    \begin{equation}
        h_i(z) = \frac{z^2}{\Big((1-\rho)z^2 + \rho\Big)^i}~\mbox{for}~i=1,2.
    \end{equation}
    Then, using \eqref{eqn:op-norm-bound}, we have
    \begin{align}
        \nonumber \|w_{k+1}-w_k\|^2 &= \eta^2 \nabla^\top g(f(w_k)) \bm{D}_k\bm{H}_k^{-2}\bm{D}_k^\top \nabla g(f(w_k))\\
        &\leq \eta^2 h_2(\lambda) \|\nabla g(f(w_k))\|^2\\
        &\leq 2\eta^2 h_2(\lambda) \frac{\mu^2}{\nu}V_k,
    \end{align}
    since $f\mapsto \nabla g(f)$ is $\mu$-Lipschitz and $\|f(w_k)-f^\star\|^2\leq \frac{2V_k}{\nu}$ by \eqref{eqn:gap-ii}. Therefore, $$\|\bm{\epsilon}_k\| \leq \frac{1}{2}\eta^2 Lh_2(\lambda)\|\nabla g(f(w_k))\|^2  \leq \eta^2 h_2(\lambda)\frac{\mu^2}{\nu}LV_k,$$ and
    \begin{align}
        \nonumber \|f(w_{k+1})-f(w_k)\|^2 &\leq 2 \|\bm{D}_k(w_{k+1}-w_k)\|^2 + 2\|\bm{\epsilon}_k\|^2\\
        &\leq 2\eta^2\|\bm{D}_k\bm{H}_k^{-1}\bm{D}_k^\top \nabla g(f(w_k))\|^2 + \eta^4 L^2V_k h_2^2(\lambda)\frac{\mu^2}{\nu}\|\nabla g(f(w_k))\|^2.
    \end{align}
    Since $\|w_k-w_0\| \leq \frac{\lambda}{L}$, we have $\bm{D}_k\bm{D}_k^\top \preccurlyeq \rr{Lip}_f^2\bm{I}$, which implies that $\bm{D}_k\bm{H}_k^{-1}\bm{D}_k^\top \preccurlyeq h_1(\rr{Lip}_f)$. Thus,
    \begin{equation}
        \|f(w_{k+1})-f(w_k)\|^2 \leq \Big(2\eta^2h_1^2(\rr{Lip}_f)+\eta^4L^2V_kh_2^2(\lambda)\frac{\mu^2}{\nu}\Big)\|\nabla g(f(w_k))\|^2.
        \label{eqn:disc-quadratic}
    \end{equation}
    On the other hand,
    \begin{equation*}
        \nabla ^\top g(f(w_k))\Big(f(w_{k+1})-f(w_k)\Big) = \gk\Big(\bm{D}_k(w_{k+1}-w_k) + \bm{\epsilon}_k\Big).
    \end{equation*}
    Note that
    \begin{align*}
        \nabla^\top g(f(w_k)) \bm{D}_k(w_{k+1}-w_k) &= -\eta \nabla^\top g(f(w_k))\bm{D}_k\bm{H}_k^{-1}\bm{D}_k^\top \nabla g(f(w_k))\\
        &\leq -\eta h_1(\lambda)\|\nabla g(f(w_k))\|^2,
    \end{align*}
    and
    \begin{align*}
        \nabla ^\top g(f(w_k))\bm{\epsilon}_k &\leq \|\nabla g(f(w_k))\|\cdot \|\bm{\epsilon}_k\|\\
        &\leq \eta^2\frac{\mu}{\sqrt{\nu}} L\sqrt{V_k} h_2(\lambda)\|\nabla g(f(w_k))\|_2^2.
    \end{align*}
    Therefore, we have
    \begin{equation}
        \nabla ^\top g(f(w_k))\Big(f(w_{k+1})-f(w_k)\Big) \leq \Big(-\eta h_1(\lambda)+\eta^2\frac{\mu}{\sqrt{\nu}} L\sqrt{V_k} h_2(\lambda)\Big)\|\nabla g(f(w_k))\|^2.
        \label{eqn:disc-linear}
    \end{equation}
    Substituting \eqref{eqn:disc-quadratic} and \eqref{eqn:disc-linear} into \eqref{eqn:lyapunov}, we obtain
    \begin{equation*}
        V_{k+1} \leq V_k + \Big(-\eta h_1(\lambda)+\eta^2\frac{\mu}{\sqrt{\nu}} L\sqrt{V_k} h_2(\lambda)+\mu\eta^2h_1^2(\rr{Lip}_f)+\eta^4L^2V_kh_2^2(\lambda)\frac{\mu^3}{\nu}\Big)\|\nabla g(f(w_k))\|^2.
    \end{equation*}
    Choose $$\eta \leq \frac{h_1(\lambda)}{6\mu h_1^2(\rr{Lip}_f)},$$ and the network width $m$ sufficiently large such that $L$ satisfies 
    \begin{equation}
        L\leq \sqrt{\frac{\nu}{V_0}}\min\left\{\frac{h_1(\rr{Lip}_f)}{h_2(\lambda)\mu\eta},\frac{h_1^2(\rr{Lip}_f)}{h_2(\lambda)}\right\}.
    \end{equation}
    Then, we have $$V_{k+1}\leq V_k-\frac{\eta h_1(\lambda)}{2}\|\nabla g(f(w_k))\|^2$$ and $L\sqrt{V_k}\leq L\sqrt{V_0}$ for all $k\in\bN$ by induction.
    From Polyak-\L{}ojasiewicz inequality, we have $\|\nabla g(f(w_k))\|^2 \geq 2\nu V_k$. Using this, we obtain
    \begin{equation}
        V_{k+1} \leq \Big(1-\eta \nu h_1(\lambda)\Big)V_k,
    \end{equation}
    for any $k < N$. Hence, for any $k \leq N$,
    \begin{equation}
        V_k \leq V_0 \Big(1-\eta\nu h_1(\lambda)\Big)^k \mbox{ and }\|f(w_k)-f^\star\|^2\leq \frac{2V_0}{\nu}\Big(1-\eta\nu h_1(\lambda)\Big)^k.
    \end{equation}
Using these inequalities, we will now show that $N = \infty$ can be established by sufficiently large $m\in\bN$. First, recall that $\|w_{k+1}-w_k\|^2 \leq \eta^2 h_2(\lambda)\|\nabla g(f(w_k))\|^2.$ Then, for $k < N$,
    \begin{align*}
        \|w_k-w_0\|_2 &\leq \sum_{s<k}\|w_{s+1}-w_s\|_2\leq \eta \sqrt{h_2(\lambda)}\sum_{s<k}\|\nabla g( f(w_s))\|_2\\
        &\leq \eta\sqrt{h_2(\lambda)}\mu\sum_{s<k}\| f(w_s)-f^\star\|_2\leq \eta\sqrt{h_2(\lambda)}\mu\sum_{s<k}\sqrt{\frac{2V_0}{\nu}}q^{s/2}\\
        &\leq \eta\sqrt{h_2(\lambda)}\mu\sqrt{\frac{2V_0}{\nu}}\frac{1}{1-\sqrt{q}}\leq 2\eta\sqrt{h_2(\lambda)}\mu\sqrt{\frac{2V_0}{\nu}}\frac{1}{1-q}\\
        &= \frac{2\sqrt{2 h_2(\lambda)}}{\nu h_1(\lambda)}\sqrt{\frac{V_0}{\nu}},
    \end{align*}
    where $q := 1-\eta\nu h_1(\lambda)$. We choose $m$ sufficiently large such that $$\frac{2\sqrt{2 h_2(\lambda)}}{\nu h_1(\lambda)}\sqrt{\frac{V_0}{\nu}} \leq \frac{\lambda}{L}.$$
    Hence, we can ensure from the above inequality that $\|w_k-w_0\|_2\leq r_0$ for all $k\in\bN$, therefore $N=\infty$. Therefore, $V_k \leq V_0 \Big(1-\eta \nu h_1(\lambda)\Big)^k$ holds for any $k\in\bN$, where we choose $m$ such that
    \begin{equation}
        L\leq \sqrt{\frac{\nu}{V_0}}\min\left\{\frac{h_1(\rr{Lip}_f)}{h_2(\lambda)\mu\eta},\frac{h_1^2(\rr{Lip}_f)}{h_2(\lambda)},\frac{\lambda\nu h_1(\lambda)}{2\sqrt{2h_2(\lambda)}}\right\}.
    \end{equation}
    Choosing $\eta = \frac{h_1(\lambda)}{6\mu h_1^2(\rr{Lip}_f)}$ yields the convergence rate
       $ V_k\leq V_0\left(1-\frac{1}{6\kappa}\frac{h_1^2(\lambda)}{h_1^2(\rr{Lip}_f)}\right)^k.$
\end{proof}

\begin{proof}[of Theorem \ref{thm:non-quad-out}]

(1) Using similar steps as Theorem \ref{thm:lm-convergence} with Sherman-Morrison-Woodbury formula, the evolution in the function space can be written as
\begin{equation*}
    \dt{f(w_t)} = -\frac{1}{\rho_t}\bm{\Sigma}_t(\bm{K}_t)\nabla g(f(w_t)),
\end{equation*}
where $\bm{\Sigma}_t(\bm{K}) := \bm{K}-\bar{\rho}_t^{-1}\bm{K}\Big(\bm{G}^{-1}(w_t)+\bar{\rho}_t^{-1}\bm{K}\Big)^{-1}\bm{K}$ for any symmetric positive semi-definite matrix $\bm{K}\in\bR^{n\times n}$. The transform $\bm{K}\mapsto \bm{\Sigma}_t(\bm{K})$ is monotonically increasing with respect to the L\"{o}wner order. Since $t < T$, we have $\bm{K}_t\succeq \lambda^2\bm{I}$, therefore $\bm{\Sigma}_t(\bm{K}_t)\succeq \bm{\Sigma}_t(\lambda^2 \bm{I})$. Denote the eigenpairs of $\bm{G}(w_t)$ as $(u_{t,i},\gamma_{t,i})$ where $\gamma_{t,i}\leq \gamma_{t,i+1}$ for all $i=1,2,\ldots,n-1$. Then, for any $i\in[n]$, $(u_i,\tilde{\gamma}_{t,i})$ is an eigenpair for $\frac{1}{\rho}\bm{\Sigma}_t(\lambda^2\bm{I})$ where $$\tilde{\gamma}_{t,i}=\frac{\lambda^2}{\rho_t+(1-\rho_t)\lambda^2\gamma_{t,i}}.$$ Using this, since $\gamma_{t,n}:=\lambda_{\max}(\bm{G}(w_t))$, we conclude that 
\begin{align*}
    \dt{g(f(w_t))} &\leq -\frac{\lambda^2}{\rho_t + (1-\rho_t)\lambda^2\gamma_{t,n}}\|\nabla g(f(w_t))\|^2\\
    &= -\frac{1}{2}\Big(\lambda^2 + \frac{1}{\gamma_{t,n}}\Big)\|\nabla g(f(w_t))\|^2,\\
    &\leq -\nu(\lambda^2 + \gamma_{t,n}^{-1})(g(f(w_t))-g(f^\star)) =: -\nu(\lambda^2 + \gamma_{t,n}^{-1})V_t,
\end{align*}
where the second line follows from the particular choice of $\rho_t$. Using Gr\"{o}nwall inequality, we obtain
\begin{align*}
    V_t &\leq V_0\exp\left(-\lambda^2\nu t - \nu\int_0^t\frac{1}{\gamma_{s,n}}\rr{d}s\right),\\
    \|f(w_t)-f^\star\| &\leq \sqrt{\frac{2V_0}{\nu}}\exp(-\frac{1}{2}\lambda^2\nu t - \frac{1}{2}\frac{\nu}{\mu}t),
\end{align*}
    for any $t < T.$ Using $\|w_t-w_0\|\leq \int_0^t\|\dot{w}_s\|\rr{d}s$ similar to Theorem \ref{thm:lm-convergence}, we conclude that $$\int_{0}^\infty\|w_s\|\rr{d}s\leq Lc^\star,$$ for some constant $c^\star =c^\star(\cD, \mu,\nu)$. Thus, if $m$ is sufficiently large, then $$Lc^\star=\frac{\sigma_2\sqrt{\sum_j\|x_j\|^4}}{\sqrt{m}}c^\star\leq r_0,$$ which implies that $T=\infty$. Note that $\int_0^\infty\|\dot{w}_s\|\rr{d}s<\infty$ implies that $\{w_t:t\geq 0\}$ is a Cauchy system in $\bR^p$, thus $w_t\rightarrow w^\star$ as $t\rightarrow\infty$ for some $w^\star\in\cB(w_0,r_0)$. Since $f(w_t)\underset{t\rightarrow\infty}{\rightarrow} f^\star$, the limit point should be an interpolant such that $f(w^\star)=f^\star$.

(2) We first make the following decomposition:
\begin{align}
    \begin{aligned}
    \bm{H}_\rho(w)-\nabla^2\cR(w^\star)&=(1-\rho(w))\Big[\underbrace{\rr{D}^\top f(w)\bm{G}(f(w))\rr{D}f(w) - \nabla^2 \cR(w)}_{(a)}\Big] \\&+ (1-\rho(w))\Big[\underbrace{\nabla^2 \cR(w)-\nabla^2 \cR(w^\star)}_{(b)}\Big]\\
    &+ \rho(w)\Big[\bm{I}-\nabla^2 \cR(w^\star)\Big].
    \end{aligned}
\end{align}
For $(a)$, we have the following identity:
\begin{equation*}
    \nabla^2\cR(w)-\rr{D}^\top f(w)\bm{G}(f(w))\rr{D}f(w) = \sum_{j=1}^n[\nabla g(f(w))]_j\nabla^2\varphi(x_j;w).
\end{equation*}
Since $\|\nabla^2\varphi(x_j;w)\|\leq \sigma_2\|x_j\|^2/\sqrt{m}$ and thus $w\mapsto \rr{D}f(w)$ is $L$-Lipschitz with $L$ defined in \eqref{eqn:L}, we obtain
\begin{equation}
    \|\nabla^2\cR(w)-\rr{D}^\top f(w)\bm{G}(f(w))\rr{D}f(w)\| \leq \mu\cdot L \cdot \rr{Lip}_f\cdot\|w-w^\star\|.
\end{equation}
For $(b)$, $C$-Lipschitz continuity of $\nabla^2\cR(w)$ implies
$$
\|\nabla^2\cR(w)-\nabla^2\cR(w^\star)\|\leq C\|w-w^\star\|.
$$
Also,
    $g(f(w))\leq \frac{\mu}{2}\|f(w)-f(w^\star)\|^2\leq \frac{\mu}{2}\rr{Lip}_f^2\|w-w^\star\|^2,$
hence $\sqrt{\cR(w)}\leq \rr{Lip}_f\sqrt{\mu/2}\|w_t-w^\star\|,$ and $$\frac{\sqrt{\cR(w)}}{{\rr{Lip}_f\sqrt{\mu/2}}+\sqrt{\cR(w)}}\leq \frac{\|w-w^\star\|}{1+\|w-w^\star\|}\leq \|w-w^\star\|.$$ Using these, we obtain the following bound:
\begin{equation}
    \|\bm{H}_\rho(w)-\nabla^2\cR(w^\star)\|\leq \Big(\mu \cdot L\cdot \rr{Lip}_f + C + 1\Big)\|w-w^\star\| =: C'\|w-w^\star\|.
\end{equation}
By Weyl's inequality \citep{horn2012matrix}, we conclude that
\begin{equation}\label{eqn:r-star-1}
    \|w-w^\star\|\leq \frac{1}{2C'}\lambda_{\min}(\nabla^2\cR(w^\star)) \Longrightarrow \lambda_{\min}(\bm{H}_\rho(w)) \geq \frac{1}{2}\lambda_{\min}(\nabla^2\cR(w^\star))\geq \frac{1}{2}\nu.
\end{equation}
For any $w\in\bR^p$, let \begin{align*}e(w):=w-w^\star,\quad \mbox{and}\quad R_1(w):= \nabla\cR(w)-\nabla^2\cR(w^\star)e(w).\end{align*} 
Using $C$-Lipschitz continuity of $\nabla^2\cR(w)$, $$\|R_1(w)\|\leq \frac{1}{2}C\|e(w)\|^2.$$ Then, letting $e_t:=e(w_t)$, we obtain
\begin{align*}
    \dot{e}_t = \dot{w}_t &= -\bm{H}_\rho^{-1}(w_t)\nabla_w\cR(w_t)\\
    &= -\bm{H}_\rho^{-1}(w_t)\nabla^2\cR(w^\star)e_t - \bm{H}_\rho^{-1}(w_t)R_1(w_t)\\
    &= -e_t - \Big([\nabla^2\cR(w^\star)]^{-1}-\bm{H}_\rho^{-1}(w_t)\Big)\nabla^2\cR(w^\star)e_t - \bm{H}_\rho^{-1}(w_t)R_1(w_t).
\end{align*}
We use the Lyapunov function $\Psi(w):=\frac{1}{2}\|e(w)\|_2^2$. Then, we have
\begin{align*}
    &\dt{\Psi(w_t)}=e_t^\top \dot{e}_t\\ &\leq  -\|e_t\|^2 + \|[\nabla^2\cR(w^\star)]^{-1}-\bm{H}_\rho^{-1}(w_t)\|\cdot\lambda_{\max}(\nabla^2\cR(w^\star))\|e_t\|^2 + \frac{C}{\lambda_{\min}(\nabla^2\cR(w^\star))}\|e_t\|^3.
\end{align*}
Using $\|A^{-1}-B^{-1}\|\leq \frac{\|A-B\|}{\lambda_{\min}(A)\lambda_{\min}(B)}$, if $\|w_t-w^\star\|\leq \frac{1}{2C'}\lambda_{\min}(\nabla^2(\cR(w^\star))$, we obtain  $$\|[\nabla^2\cR(w^\star)]^{-1}-\bm{H}_\rho^{-1}(w_t)\|\leq \frac{2C'}{\lambda_{\min}^2(\nabla^2\cR(w^\star))}\|e_t\|.$$ Hence, with $C'' := \frac{2C'\lambda_{\max}(\nabla^2\cR(w^\star))}{\lambda^2_{\min}(\nabla^2\cR(w^\star))} + \frac{C}{\lambda_{\min}(\nabla^2\cR(w^\star))},$ we have
\begin{align*}
    \dt{\Psi(w_t)} \leq -2\Psi(w_t) + C''[\Psi(w_t)]^{3/2}.
\end{align*}
If 
\begin{equation}\label{eqn:r-star-2}\|w_t-w^\star\|\leq \min\left\{\frac{\sqrt{2}}{C''}, \frac{1}{2C'}\lambda_{\min}(\nabla^2\cR(w^\star))\right\}=:r^\star,\end{equation}
then $C''[\Psi(w_t)]^{1/2}\leq 1,$ and thus $$\dt{\Psi(w_t)}\leq -\Psi(w_t),\quad t\geq \tau_{r^\star}.$$ We conclude the proof by using Grönwall's lemma.
\end{proof}

\begin{proof}[of Proposition \ref{prop:implicit-bias}]
    Under Assumption \ref{assumption:ntk}, we have $\rr{D}f(w_0)\bm{H}_\rho^{-1}(w_0)\rr{D}^\top f(w_0) \succ 0$, thus $\lim_{t\rightarrow\infty}\bar{f}_0(u_t) = y$ using the same Grönwall argument as Theorem \ref{thm:lm-convergence}. Let $$\beta_t = -\int_0^t(y-\bar{f}_0(u_s))\rr{d}s,\quad t \geq 0.$$ Then, 
    $$
        u_t-w_0 = \bm{H}_\rho^{-1}(w_0)\rr{D}^\top f(w_0)\beta_t,\quad t \geq 0.
    $$
    Since $\bar f_0(u_t)\rightarrow y$ as $t\rightarrow\infty$, $$\lim_{t\rightarrow\infty}\rr{D}f(w_0)\bm{H}_\rho^{-1}(w_0)\rr{D}^\top f(w_0)\beta_t = y\quad\Rightarrow\quad \lim_{t\rightarrow\infty}\beta_t = (\rr{D}f(w_0)\bm{H}_\rho^{-1}(w_0)\rr{D}^\top f(w_0))^{-1}y.$$ Therefore, $$\lim_{t\rightarrow\infty}\{u_t - w_0\} = \bm{H}_\rho^{-1}(w_0)\rr{D}^\top f(w_0) \Big(\rr{D}f(w_0)\bm{H}_\rho^{-1}(w_0)\rr{D}^\top f(w_0)\Big)^{-1}y.$$
    By Karush-Kuhn-Tucker (KKT) conditions \citep{bach2024learning, boyd2004convex}, the limit yields $u^\star-w_0$ with minimum $\bm{H}_\rho(w_0)$-norm subject to the affine constraints $\rr{D}f(w_0)(u-w_0)=y$.
\end{proof}

\begin{proof}[of Lemma \ref{lemma:par-dev-deep}]
    The proof extends Lemma B.4 in \cite{du2019gradient} with an improved dependence on $n$ and $\delta$ for smooth and bounded activations. Let $\mathcal{F}_h=\sigma(W_0^{(1)},\ldots,W_0^{(h)})$ for any $h\in[H]$. For notational simplicity, we denote $\bm{x}_j^{(h)}(\bm{W}_0)$ as $\bm{x}_j^{(h)}$ throughout the proof. For any data point $j\in[n]$, $$\|\bm{x}_j^{(h)}\|_2^2=\frac{a_\sigma}{m}\sum_{i=1}^m\sigma^2\big(\langle W_{0,i}^{(h)},\bm{x}_j^{(h-1)}\rangle\big),$$
    where $W_{0,i}^{(h)}$ is the $i$-th row of $W_0^{(h)}$. Then, $\bm{x}_j^{(h)}(\bm{W}_0)$ is $\mathcal F_h$-measurable, each summand in the above display is bounded by $a_\sigma \sigma_0^2$ almost surely, and $$\mathbb{E}[\|\bm{x}_j^{(h)}\|_2^2|\mathcal{F}_{h-1}]=a_\sigma\bE[\sigma^2(\langle W_{0,1}^{(h)},\bm{x}_j^{(h-1)}\rangle)|\mathcal{F}_{h-1}].$$ Thus, by Hoeffding's inequality, conditioned on $\mathcal F_{h-1}$, we have
    \begin{equation}\label{eqn:conc-ineq}
        \max_{j=1,2,\ldots,n}\Big|\|\bm{x}_j^{(h)}\|_2^2 - a_\sigma \bE\big[\sigma^2\big(\big\langle W_{0,1}^{(h)},\bm{x}_j^{(h-1)}\big\rangle\big)\big|\mathcal F_{h-1}\big]\Big| \leq a_\sigma\sigma_0^2\sqrt{\frac{\log(2n/\delta)}{2m}},
    \end{equation}
    with probability at least $1-\delta$. Since $\bE_{u_0\sim\cN(0,\bm{I})}[\sigma^2(u_0^\top x)]=\bE_{z\sim\cN(0,1)}[\sigma^2(z\|x\|_2)]$, we have
    \begin{equation}
            \bE\big[\sigma^2\big(\big\langle W_{0,1}^{(h)},\bm{x}_j^{(h-1)}\big\rangle\big)\big|\mathcal F_{h-1}\big] = \bE_{z\sim\cN(0,1)}\big[\sigma^2\big(z\cdot \|\bm{x}_j^{(h-1)}\|_2\big)\big|\mathcal F_{h-1}\big].
            \label{eqn:deep-i}
    \end{equation}
    Using $\sigma_0$-boundedness and $\sigma_1$-Lipschitz continuity of $\sigma$, for any $x\in\bR^m$,
    $$
        \Big|\bE[\sigma^2(z)] - \bE[\sigma^2(z\cdot\|x\|_2)]\Big| \leq 2\sigma_0\sigma_1\Big|1-\|x\|_2\Big|.
    $$
    Hence, dividing both sides by $a_\sigma$, we obtain
    \begin{equation}\label{eqn:deep-ii}
        \Big|1 - a_\sigma\bE[\sigma^2(z\cdot\|\bm{x}_j^{(h-1)}\|_2)]\Big| \leq C\Big|1-\|\bm{x}_j^{(h-1)}\|_2\Big|.
    \end{equation}
    Using \eqref{eqn:conc-ineq}, \eqref{eqn:deep-i} and \eqref{eqn:deep-ii}, we obtain
    $$
        \Big|1-\|\bm{x}_j^{(h)}\|_2\Big|\leq \Big|1-\|\bm{x}_j^{(h)}\|_2^2\Big| \leq C\Big|1-\|\bm{x}_j^{(h-1)}\|_2\Big| + a_\sigma\sigma_0^2\sqrt{\frac{\log(2n/\delta)}{2m}}.
    $$
    With the choice $$m \geq \frac{a_\sigma^2\sigma_0^4C^2\log(2nH/\delta)}{2k^2},$$ the last term above is bounded by $\frac{k}{C}$ and we obtain the recursion
    $$
        \Big|1-\|\bm{x}_j^{(h)}\|_2\Big| \leq C\Big|1-\|\bm{x}_j^{(h-1)}\|_2\Big|+\frac{k}{C},
    $$
    for all $j\in[n]$ and $h\in[H]$ with probability at least $1-\delta$. With the choice $k=\frac{C(C-1)}{2(C^H-1)}$, this implies
    \begin{equation}
        \frac{1}{2} \leq  \|\bm{x}_j^{(h)}\|_2 \leq \frac{3}{2},\quad \forall j\in[n],\forall h\in[H],
    \end{equation}
    with probability at least $1-\delta$ at initialization. The bound on $\max_{h\in[H]}\|W_0^{(h)}\|_2$ follows from sub-Gaussian random matrix inequality (Theorem 4.4.3) in \cite{vershynin2018high}. Since $\|c_0\|_2 \leq \sqrt{m}$ and $\|c_0\|_4 \leq m^{1/4}$ almost surely at initialization, the result follows from substituting these inequalities into Lemma B.4 in \cite{du2019gradient}.
\end{proof}

    \section{Omitted Proofs from Section \ref{sec:up}}\label{app:up}
The following lemma, which characterizes the magnitude of the second fundamental form and quantifies curvature, will be fundamental throughout the analysis.

\begin{lemma}[Curvature bounds]\label{lemma:curvature}
    Take an arbitrary $\alpha f(w)\in\cM$ and $u\in\cT_{\alpha f(w)}\cM\backslash\{0\}$. Let $\beta:[0,1]\rightarrow\cM$ be any smooth curve such that $\beta_0=\alpha f(w)$ and $\dt{\beta_t}\big|_{t=0}=u$. Since $\alpha f$ is a smooth injective immersion, we have a smooth curve $\gamma:[0,1]\rightarrow B$ such that $\beta_t:=\alpha f(\gamma_t)\in\cM$ with $\gamma_0=w$ and $\dt{\alpha f(\gamma_t)}\big|_{t=0}=u$. Then, we have
    \begin{equation}
        \Big\|\dt{}\Big[\bm{P}(\alpha f(\gamma_t))\Big]\Big|_{t=0}\Big\| \leq \frac{2L\|u\|_2}{\alpha \lambda_0^2}.
    \end{equation}
\end{lemma}
    \begin{proof}[of Lemma \ref{lemma:curvature}]
        In the following, we will establish an upper bound on $\big\|\big[\dt{}\bm{P}(\alpha f(\gamma_t))\big]\big|_{t=0}\|$. The first inequality follows from a classical result in perturbation theory.
    \begin{claim}[Theorem 3.9 in \cite{stewart1990matrix}]
        Let $J_i\in\bR^{n\times p},~i=1,2,$ be two matrices such that $$\min\{\lambda_{\min}(J_1^\top J_1), \lambda_{\min}(J_2^\top J_2)\}\geq \lambda^2.$$ Let $P_i=J_i[J_i^\top J_i]^{-1}J_i^\top,~i=1,2$. Then, we have 
        \begin{equation}
            \|P_1-P_2\| \leq \frac{2}{\lambda}\|J_1-J_2\|.
        \end{equation}
        \label{claim:perturbation}
    \end{claim}
    \noindent For any $t\in[0,1]$, we have $\gamma_t\in B$, thus $\lambda_{\min}\big(\rr{D}^\top f(\gamma_t)\rr{D}f(\gamma_t)\big)\geq \lambda_0^2$ by Lemma \ref{lemma:min-eig}. Thus, for any $t\in[0,1]$, we have
    \begin{equation*}
        \|\bm{P}(\alpha f(\gamma_t))-\bm{P}(\alpha f(\gamma_0))\|\leq \frac{2}{\lambda_0}\|\rr{D}f(\gamma_t)-\rr{D}f(\gamma_0)\|.
    \end{equation*}
    Recall that $w\mapsto \rr{D}f(w)$ is globally $L$-Lipschitz where $L$ is explicitly given in \eqref{eqn:L}. Therefore,
    \begin{equation}\label{eqn:hessian-2a}
        \|\bm{P}(\alpha f(\gamma_t))-\bm{P}(\alpha f(\gamma_0))\|\leq \frac{2L}{\lambda_0}\|\gamma_t-\gamma_0\|.
    \end{equation}
    By Lemma \ref{lemma:injection}, $\alpha f|_B:B\mapsto \cM$ is a bijective mapping, hence its inverse $f_\alpha^{-1}:\cM\rightarrow B$ such that $f_\alpha^{-1}(\alpha f(w))=w,~w\in B$ exists. Furthermore, $f_\alpha:\cM\rightarrow B$ is Lipschitz continuous with 
        \begin{equation}
            \sup_{y\in \cM}\|\rr{D}f_\alpha^{-1}(z)\| \leq \frac{1}{\alpha}\cdot\frac{1}{\lambda_0},
        \end{equation}
        since $\rr{D}f_\alpha^{-1}(z)=[\rr{D}\alpha f(f_\alpha^{-1}(z))]^+$ for any $z\in\cM$ and $\lambda_{\min}(\rr{D}^\top f(w)\rr{D}f(w))\geq \lambda_0$ for all $w\in B$.
        Hence, we have
    $$
    \|\gamma_t-\gamma_0\|=\|f_\alpha^{-1}(\alpha f(\gamma_t))-f_{\alpha}^{-1}(\alpha f(\gamma_0))\|\leq \frac{1}{\alpha\lambda_0}\|\alpha f(\gamma_t)-\alpha f(\gamma_0)\|.
    $$
    Substituting this into \eqref{eqn:hessian-2a}, we obtain
    $$
    \|\bm{P}(\alpha f(\gamma_t))-\bm{P}(\alpha f(\gamma_0)\|\leq \frac{2L}{\alpha \lambda_0^2}\|\alpha f(\gamma_t)-\alpha f(\gamma_0)\|.
    $$
    Therefore, we have
    \begin{equation}\label{eqn:hessian-2b}
        \Big\|\dt{}\Big[\bm{P}(\alpha f(\gamma_t))\Big]\Big|_{t=0}\Big\| \leq \frac{2L}{\alpha \lambda_0^2}\cdot\Big\|\dt{\alpha f(\gamma_t)}\Big|_{t=0}\Big\|_2=\frac{2L\|u\|_2}{\alpha \lambda_0^2}
    \end{equation}
    since $\dt{\alpha f(\gamma_t)}\big|_{t=0}=u$.
    \end{proof}

\begin{proof}[of Lemma \ref{lemma:g-convexity-M}]
    Take any $\alpha f(w)\in\cM$ and $u\in\cT_{\alpha f(w)}\cM\backslash\{0\}$, and let
    \begin{equation}
        \mathscr{L}(\alpha f(w),u):= \Big\|\dt{}\Big[\bm{P}(\alpha f(\gamma_t))\Big]\Big|_{t=0}\Big\|
    \end{equation}
    be the norm of the differential in the direction of $u$. We have $L(\alpha f(w), u) = \|II_{\alpha f(w)}(u, \cdot)\|_\rr{op}$. Let 
    \begin{equation*}
        \Lambda := \sup_{w\in B}~\sup_{u\in \rr{Im}(\rr{D}f(w)):\|u\|=1}~\mathscr{L}(\alpha f(w), u).
    \end{equation*}
    By the Alexander-Berg-Bishop Characterization Theorem \citep{alexander1993geometric}, if
    \begin{equation}\label{eqn:cat-condition}
        \sup_{w,w'\in B}\|\alpha f(w)-\alpha f(w')\| \leq {\pi/\Lambda},
    \end{equation}
    then $(\cM,\langle \cdot,\cdot\rangle^\cM)$ is a $\rr{CAT}(\Lambda^2)$ space; hence there exists a geodesic segment $c:[0,1]\rightarrow \cM$ such that $c(0)=\alpha f(w)$ and $c(1)=\alpha f(w')$ by \cite{BridsonHaefliger1999}, Proposition 1.4 in Part II, implying the geodesic convexity of $\cM$.
    
    By Lemma \ref{lemma:curvature}, we have $\Lambda \leq \frac{2L}{\alpha \lambda_0^2}$ for each $w,w'\in B$. Furthermore, $$\sup_{w,w'\in B}\|\alpha f(w)-\alpha f(w')\|\leq 2\alpha r_0 \rr{Lip}_f.$$ Hence, a sufficient condition for \eqref{eqn:cat-condition} is \begin{equation}\label{eqn:suff-condition-cat}r_0 \leq \frac{\lambda_0^2\pi}{4L\rr{Lip}_f}\end{equation} due to $$\sup_{w,w'\in B}\|\alpha f(w)-\alpha f(w')\| \leq 2\alpha r_0 \rr{Lip}_f ~\mbox{  and  }~ \frac{\alpha \pi\lambda_0^2}{2L}\leq \frac{\pi}{\Lambda}.$$ Since $r_0\leq \frac{\lambda_0^2}{4L\rr{Lip}_f}\frac{\nu}{\mu}\leq \frac{\lambda_0^2}{4L\rr{Lip}_f}$ (see \eqref{eqn:r0}), the sufficient condition \eqref{eqn:suff-condition-cat} holds, thus $\cM$ is geodesically convex.
\end{proof}
    
\begin{proof}[of Theorem \ref{thm:geodesic-regularity}]
    (a) Fix $\alpha f(w)\in\cM$ and $u\in\cT_{\alpha f(w)}\cM\backslash\{0\}$. Let $\beta:[0,1]\rightarrow\cM$ be any smooth curve such that $\beta_0=\alpha f(w)$ and $\dt{\beta_t}\big|_{t=0}=u$. Since $\alpha f$ is a smooth injective immersion, we have a smooth curve $\gamma:[0,1]\rightarrow B$ such that $\beta_t:=\alpha f(\gamma_t)\in\cM$ with $\gamma_0=w$ and $\dt{\alpha f(\gamma_t)}\big|_{t=0}=u$. Then, the quadratic form for the Riemannian Hessian is
    \begin{align}\label{eqn:hessian}
        \langle u, \rr{Hess}~g(\alpha f(w))[u]\rangle_{\alpha f(w)}^\cM &= u^\top \lim_{t\rightarrow 0}\frac{\bm{P}(\alpha f(w))\Big[\grad_{\alpha f(\gamma_t)}^\cM g(\alpha f(\gamma_t))-\grad_{\alpha f(w)}^\cM g(\alpha f(w))\Big]}{t}
    \end{align}
    by (5.19) in \cite{boumal2023introduction}. Recall that $$\grad_{\alpha f(\gamma_t)}^\cM g(\alpha f(\gamma_t))=\bm{P}(\alpha f(\gamma_t))\nabla g(\alpha f(\gamma_t)).$$ Thus, we can make the following decomposition for any $t \in [0,1]$:
    \begin{align}
        \begin{aligned}\bm{P}(\alpha f(w))&\Big[\grad_{\alpha f(\gamma_t)}^\cM g(\alpha f(\gamma_t))-\grad_{\alpha f(w)}^\cM g(\alpha f(w))\Big] \\ &= \bm{P}(\alpha f(w))\Big[\nabla g(\alpha f(\gamma_t))-\nabla g(\alpha f(w))\Big]+\Big[\bm{P}(\alpha f(\gamma_t))-\bm{P}(\alpha f(w))\Big]\nabla g(\alpha f(\gamma_t)).
        \label{eqn:hessian-decomposition}
        \end{aligned}
    \end{align}

    The first term can be lower bounded as
    \begin{align}
        \nonumber u^\top \lim_{t\downarrow 0}\frac{\bm{P}(\alpha f(w))\Big[\nabla g(\alpha f(\gamma_t))-\nabla g(\alpha f(w))\Big]}{t} &= u^\top \lim_{t\downarrow 0}\frac{\nabla g(\alpha f(\gamma_t))-\nabla g(\alpha f(\gamma_0))}{t}\\
        \nonumber &= u^\top \nabla^2 g(\alpha f(\gamma_0))\dt{\alpha f(\gamma_t)}\Big|_{t=0}\\
        & \geq \nu\|u\|_2^2=\nu(\|u\|_{\alpha f(w)}^\cM)^2,
        \label{eqn:hessian-1}
    \end{align}
    since $h\mapsto g(h)$ is (Euclidean) $\nu$-strongly convex, thus $\nabla^2 g(h)\succeq \nu I$ for all $h\in\bR^n$, and $u\in\mathcal{T}_{\alpha f(w)}\cM$, thus $\bm{P}(\alpha f(w))u=u$.  
    
    For the second term, let
    \begin{equation*}
        \rr{Lip}_g^\cC:=\sup_{z\in \cC}\|\nabla g(z)\|,
    \end{equation*}
    for any $\cC\subset \bR^n$.
    Then, since $\gamma_t \in \cM$ for all $t\in[0,1]$, we have
    \begin{align}
        \nonumber u^\top \lim_{t\downarrow 0}\frac{[\bm{P}(\alpha f(\gamma_t))-\bm{P}(\alpha f(w))]\nabla g(\alpha f(\gamma_t))}{t} &\geq -\rr{Lip}_g^\cM\cdot \|u\|_2\cdot \lim_{t\downarrow 0}\frac{\|\bm{P}(\alpha f(\gamma_t))-\bm{P}(\alpha f(w))\|}{t}\\
        &=-\rr{Lip}_g^\cM\cdot\|u\|_2\cdot\Big\|\dt{}\Big[\bm{P}(\alpha f(\gamma_t))\Big]\Big|_{t=0}\Big\|.\label{eqn:shape-curvature}
    \end{align}
     In order to bound the last term above, we use Lemma \ref{lemma:curvature}. Substituting \eqref{eqn:hessian-2b} into \eqref{eqn:shape-curvature}, we obtain
    \begin{equation}\label{eqn:hessian-2}
        u^\top \lim_{t\downarrow 0}\frac{[\bm{P}(\alpha f(\gamma_t))-\bm{P}(\alpha f(w))]\nabla g(\alpha f(\gamma_t))}{t} \geq -\frac{2L\cdot\rr{Lip}_g^\cM}{\alpha\lambda_0^2}\cdot \|u\|_2^2.
    \end{equation}
    Substituting \eqref{eqn:hessian-1} and \eqref{eqn:hessian-2} into \eqref{eqn:hessian} by using the decomposition \eqref{eqn:hessian-decomposition}, we conclude that
    \begin{equation}\label{eqn:hessian-lower-bound}
        \langle u, \rr{Hess}~g(\alpha f(w))[u]\rangle_{\alpha f(w)}^\cM \geq \nu\|u\|_2^2-\frac{2L\cdot\rr{Lip}_g^\cM}{\alpha\lambda_0^2}\cdot \|u\|_2^2.
    \end{equation}
    Finally, for any $w\in B$, we have 
    \begin{align*}
    \|\nabla g(\alpha f(w))\|_2 &= \|\nabla g(\alpha f(w))-\nabla g(f^\star)\|\\
    &= \mu \alpha \|f(w)-f(w_0)\| + \mu\|f^\star\|\\
    &\leq \mu\alpha \rr{Lip}_f\|w-w_0\| + \mu\|f^\star\| \leq \mu\Big( \alpha \rr{Lip}_fr_0 + \|f^\star\|\Big),
    \end{align*}
    which follows from $\mu$-Lipschitz continuity of $\nabla g$, $\nabla g(f^\star)=0$, and $f(w_0)=0$. Hence, we have
    \begin{equation}
        \rr{Lip}_g^\cM \leq \mu\Big( \alpha \rr{Lip}_fr_0 + \|f^\star\|\Big).
    \end{equation}
    By choosing $\alpha$ as stated in the theorem, we ensure that $\frac{2L\mu \|f^\star\|}{\alpha\lambda_0^2}\leq \frac{\nu}{2}$. Since $r_0 \leq \frac{\nu\lambda_0^2}{4\mu L \rr{Lip}_f}$, we also have $\frac{2L\mu\rr{Lip}_f r_0}{\lambda_0^2}\leq \frac{\nu}{2}$. Using these two results, \eqref{eqn:hessian-lower-bound} implies the positive semi-definiteness of the Riemannian Hessian on $\cM$, which implies geodesic convexity of $g|_\cM$ since $\cM$ is a geodesically convex set.

    (b) $g|_\cM:\cM\rightarrow\bR$ is a geodesically convex set, and $\cS:=\{z\in\cM:g(z)\leq g(0)\}$ is a sublevel set for $g$. Then, $\cS$ is geodesically convex by Proposition 11.8 in \cite{boumal2023introduction}.

    (c) Fix $\alpha f(w)\in\cS$ and $u\in\cT_{\alpha f(w)}\cM\backslash\{0\}$. Let $\beta:[0,1]\rightarrow\cS$ be any smooth curve such that $\beta_0=\alpha f(w)$ and $\dt{\beta_t}\big|_{t=0}=u$, and let $\gamma:[0,1]\rightarrow B$ such that $\beta_t:=\alpha f(\gamma_t)\in\cS$ with $\gamma_0=w$ and $\dt{\alpha f(\gamma_t)}\big|_{t=0}=u$. The proof of geodesic strong convexity of $g$ in $\cS$ follows identical steps as (a) until \eqref{eqn:hessian-lower-bound}, where $\rr{Lip}_g^\cS$ replaces $\rr{Lip}_g^\cM$ since $\alpha f(\gamma_t)\in\cS$ for all $t\in[0,1]$. Now, note that
    \begin{align*}
        \|\nabla g(\alpha f(w))\|&\leq \mu\|\alpha f(w)-f^\star\| \leq \sqrt{\frac{2(g(\alpha f(\gamma_t))-g(f^\star))}{\nu}}\\
        &\leq \sqrt{\frac{2g(0)}{\nu}}
    \end{align*}
    since $\alpha f(\gamma_t)\in\cS$, thus $g(\alpha f(\gamma_t))\leq g(0)$. This implies that $\rr{Lip}_g^\cS \leq \sqrt{\frac{2g(0)}{\nu}}$. 


    (d) Similar to the proof of (c), fix $\alpha f(w)\in\cS$ and $u\in\cT_{\alpha f(w)}\cM\backslash\{0\}$. Let $\beta:[0,1]\rightarrow \cS$ be any smooth curve such that $\beta_t=\alpha f(w)$ and $\dt{\beta_t}\big|_{t=0}=u$. We aim to find an upper bound for the quadratic form in \eqref{eqn:hessian} using the decomposition in \eqref{eqn:hessian-decomposition}. Similar to \eqref{eqn:hessian-1}, we have
    \begin{align}
        \nonumber
        u^\top \lim_{t\downarrow 0}\frac{\bm{P}(\alpha f(w))\Big[\nabla g(\alpha f(\gamma_t))-\nabla g(\alpha f(w))\Big]}{t}&= u^\top \nabla ^2g(\alpha f(w))\dt{\alpha f(w_t)}\Big|_{t=0}\\
        \nonumber &= u^\top\nabla^2 g(\alpha f(w))u\\
        \label{eqn:hessian-1u} &\leq \mu\|u\|_2^2,
    \end{align}
    where the second line holds since $\dt{\alpha f(\gamma_t)}\big|_{t=0}=u$ and the inequality is due to $\sup_{h\in\bR^n}\|\nabla^2 g(h)\|\leq \mu$ from the Lipschitz continuity of $\nabla g$. We also have \begin{align}
        \nonumber u^\top \lim_{t\downarrow 0}\frac{[\bm{P}(\alpha f(\gamma_t))-\bm{P}(\alpha f(w))]\nabla g(\alpha f(\gamma_t))}{t} &\leq \rr{Lip}_g^\cS\cdot \|u\|_2\cdot \lim_{t\downarrow 0}\frac{\|\bm{P}(\alpha f(\gamma_t))-\bm{P}(\alpha f(w))\|}{t}\\
        &=\rr{Lip}_g^\cS\cdot\|u\|_2\cdot\Big\|\dt{}\Big[\bm{P}(\alpha f(\gamma_t))\Big]\Big|_{t=0}\Big\|.\label{eqn:shape-curvature-u}
    \end{align}
    Substituting \eqref{eqn:hessian-2b} into the above inequality, we obtain
    \begin{align}
        \nonumber u^\top \lim_{t\downarrow 0}\frac{[\bm{P}(\alpha f(\gamma_t))-\bm{P}(\alpha f(w))]\nabla g(\alpha f(\gamma_t))}{t} \leq \frac{2L\cdot\rr{Lip}_g^\cS}{\alpha \lambda_0^2}\cdot \|u\|_2^2 &\leq \frac{\nu}{2}\|u\|_2^2\\
        &\leq \frac{\mu}{2}\|u\|_2^2,
        \label{eqn:hessian-2u}
    \end{align}
    where the last inequality holds since $\nu I\preccurlyeq \nabla^2g(h)\preccurlyeq\mu I$ for all $h\in\bR^n$. From \eqref{eqn:hessian-1u} and \eqref{eqn:hessian-2u}, the decomposition in \eqref{eqn:hessian-decomposition} implies $$\langle u, \rr{Hess}~g(\alpha f(w))[u]\rangle_{\alpha f(w)}^\cM\leq \frac{3\mu}{2}\|u\|_2^2,$$ which concludes the proof.
    
\end{proof}

\begin{proof}[of Lemma \ref{lemma:riemannian-pl}]
    Note that $\cS$ is a geodesically convex set, thus for any $w_1,w_2\in B$, there exists a smooth curve $c:[0,1]\rightarrow \cM$ and a tangent vector $\tilde{v}\in\cT_{\alpha f(w_1)}\cM$ such that
\begin{equation*}
    c(0)=\alpha f(w_1),~c(1)=\alpha f(w_2),\mbox{ and }c(\xi)=\Exp_{\alpha f(w_1)}(\xi\tilde{v})\in\cS\mbox{ for }\xi\in[0,1].
\end{equation*}
Using this, for any $t < T$, there exists a smooth curve $c:[0,1]\rightarrow \cM$ and a tangent vector $v_t\in\cT_{\alpha f(w^\star)}\cM$ such that
\begin{equation*}
    c(0)=\alpha f(w^\star),~c(1)=\rr{Exp}_{\alpha f(w^\star)}(v_t)=\alpha f(w_t),~\mbox{and }c(\xi)\in\cS\mbox{ for all }\xi \in [0,1].
\end{equation*}
For any $(\alpha f(w),u)$ in the tangent bundle, let $\gamma(\xi):=\rr{Exp}_{\alpha f(w)}(\xi u)$ be corresponding geodesic. Then, let $P_{\xi u}:=\rr{PT}_{\xi\leftarrow 0}^\gamma$ be the parallel transport from $\alpha f(w)$ to $\rr{Exp}_{\alpha f(w)}(\xi u)$ along $\gamma$. $P_u^{-1}$ is an isometry \cite{lee2018introduction}. Since $g|_\cS$ has geodesically $\frac{3\mu}{2}$-Lipschitz continuous gradients by Theorem \ref{thm:geodesic-regularity}(c), we have
\begin{align*}
    \|P_{v_t}^{-1}\grad_{\alpha f(w_t)}^\cM g(\alpha f(w_t))-\grad_{\alpha f(w^\star)}^\cM g(\alpha f(w^\star))\| &\leq \frac{3\mu}{2}\|v_t\|
\end{align*}
by Prop. 10.53 in \cite{boumal2023introduction}. Since $\grad_{\alpha f(w^\star)}^\cM g(\alpha f(w^\star))=0$ and $P_{v_t}^{-1}$ is an isometry, we have $$\|\grad_{\alpha f(w_t)}^\cM g(\alpha f(w_t))\| \leq \frac{3\mu}{2}\|v_t\|.$$ Theorem \ref{thm:geodesic-regularity} implies that $g|_\cS$ is $\frac{\nu}{2}$-geodesically strongly convex, therefore
\begin{align*}
    \frac{\nu}{4}\|v_t\|^2 &\leq g(\alpha f(w_t))-g(\alpha f(w^\star))\\
    &\leq \frac{1}{\nu}\|\grad_{\alpha f(w_t)}^\cM g(\alpha f(w_t))\|_{\alpha f(w_t)}^2.
\end{align*}
where the second inequality follows from the Riemannian counterpart of the Polyak-\L{}ojasiewicz inequality \cite{boumal2023introduction}.
\end{proof}
    
\section{Numerical Experiments}\label{sec:num}
We investigate the numerical performance of the Gauss-Newton gradient flow in the over- and underparameterized settings with two ill-conditioned regression problems. In both problems, we use the loss function $g(\psi)=\frac{1}{2n}\|\psi-y\|_2^2$ where $y=[y_1,y_2,\ldots,y_n]^\top\in\bR^n$. The code to re- produce the experiments can be found in the repository \url{https://github.com/semihcayci/gauss-newton}.

\textbf{Single index model.} We consider a single-index model with a training set $\cD=\{(x_j,y_j)\in\bR^d\times\bR:j=1,2,\ldots,n\}$ where the input is $x_j\sim_\rr{iid}\rr{Unif}(\bS^{d-1})$ and the label is 
\begin{equation}
    y_j = \rr{ReLU}(u^\top x_j)+\epsilon_j,
    \label{eqn:single-index-model}
\end{equation}
where $\rr{ReLU}(z)=\max\{0,z\}$, $u\in\bR^d$ is the target direction and $\epsilon_j\sim\cN(0,1)$ is the noise for $j=1,2,\ldots,n$. As noted in Remark \ref{remark:geometry}, this input distribution leads to small $\lambda_{\min}(\bm{K}_0)$.

\textbf{California Housing dataset.} In the second set of experiments, we consider the California Housing dataset $\mathcal{D}:=\{(x_j,y_j)\in\bR^d\times\bR:j=1,2,\ldots,n\}$ \cite{pace1997sparse}, where each feature vector $x_j\in\bR^d$ with $d = 8$ represents normalized housing-related attributes, and $y_j\in\bR$ represents median house value for $j=1,2,\ldots,n$. We randomly subsample $n$ data points for training. 

\subsection{Overparameterized regime}
 We consider an overparameterized problem with $n\gg p$ in Figure \ref{fig:overpar}. For the single-index model, we use a dataset of $n=800$ samples of ambient dimension $d=16$ and a $\tanh$ neural network with $p=10800$ parameters. For the California Housing dataset, we randomly subsample a training set of size $n=800$, and use a $\tanh$ neural network with $p=6400$ parameters. The parameters are trained by using the Gauss-Newton method with various regularization choices: 
\begin{enumerate}[label=(\roman*)]
    \item adaptive damping $\rho_t = \frac{\frac{1}{4}\lambda_t^2}{1+\frac{1}{4}\lambda_t^2}$,
    \item constant data-based damping with $\rho_t=\frac{\lambda^2}{\lambda^2+1}$ where $\lambda_{\min}(\bm{K}_0) = 4\lambda^2\bm{I}$,
\end{enumerate}
Note that $\rho_t=1.0$ corresponds to the (non-preconditioned) gradient flow. The continuous-time dynamics are simulated by using Euler's method with $\Delta  t = 0.01$.

 \begin{figure}[ht!]
    \centering
    \begin{subfigure}[t]{0.5\textwidth}
        \centering
        \includegraphics[width=.9\linewidth]{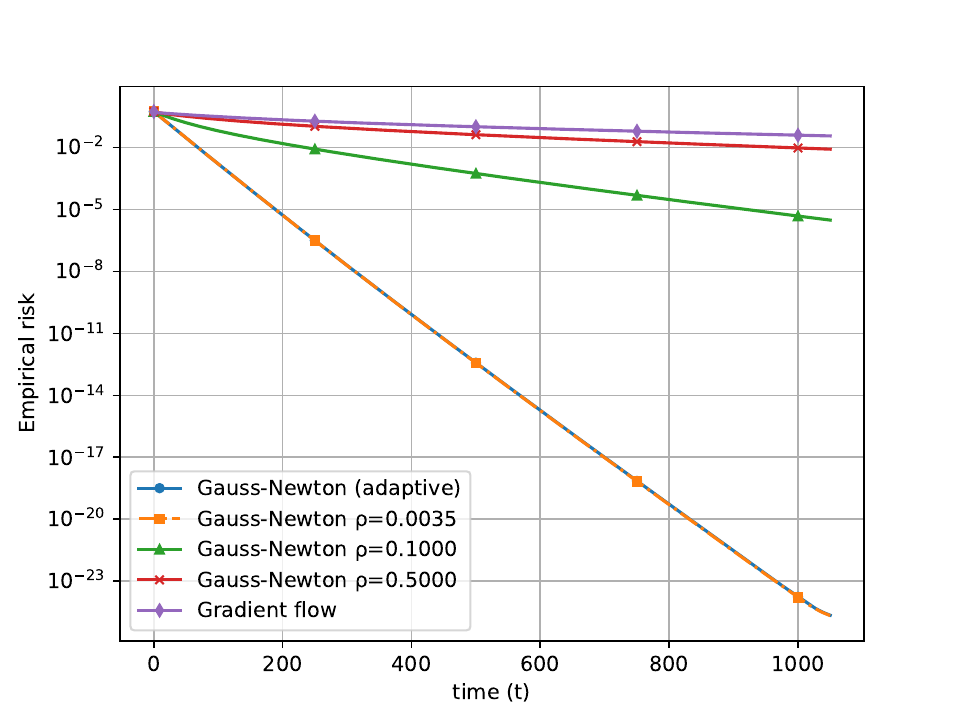}
        \caption{Single-index model}
        \label{fig:overpar-si}
    \end{subfigure}%
    ~ 
    \begin{subfigure}[t]{0.5\textwidth}
        \centering
        \includegraphics[width=.9\linewidth]{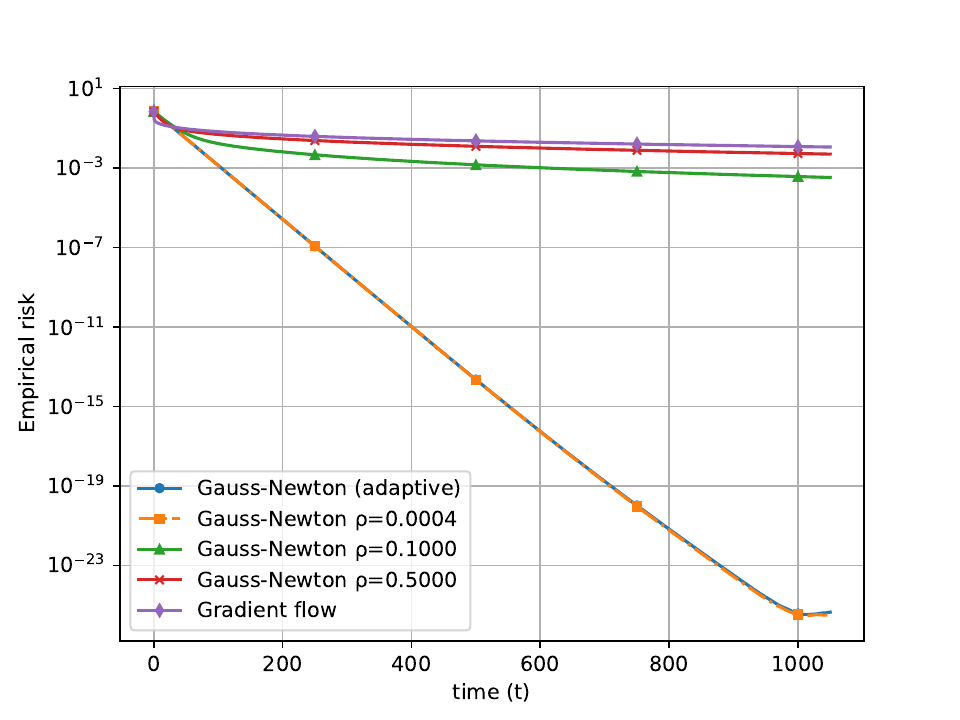}
        \caption{California Housing}
        \label{fig:overpar-ch}
    \end{subfigure}
    \caption{Empirical risk in the overparameterized regime under the Gauss-Newton dynamics with various regularization schemes $\rho=(\rho_t)_{t\geq 0}$. Gradient flow $(\rho_t=1)$ suffers from slow convergence due to the ill-conditioned neural tangent kernel, while the Gauss-Newton with appropriate constant or adaptive damping schedules achieve fast exponential convergence rates.}
    \label{fig:overpar}
\end{figure}
 
In these examples, the neural tangent kernel $\bm{K}_0$ is ill-conditioned (see also Remark \ref{remark:geometry}), thus the gradient flow suffers from slow convergence, while the Gauss-Newton method with appropriate constant and adaptive damping choices achieve fast convergence. In particular, the adaptive choice $\rho_t=\lambda_t^2/(1+\lambda_t^2)$ and the constant data-dependent choice $\rho_t=\lambda^2/(1+\lambda^2)$ achieves fast linear convergence rate, verifying the theoretical results in Theorem \ref{thm:lm-convergence}. Note that in the lazy training regime with large $\alpha\sqrt{m}$ that we consider, we have $\lambda_t^2/(1+\lambda_t^2)\gtrapprox \lambda^2/(1+\lambda^2)$, thus adaptive and data-dependent constant damping choices yield very similar empirical risk performance as characterized in Theorem \ref{thm:lm-convergence}.



\subsection{Underparameterized Regime}

We investigate the performance of Gauss-Newton dynamics (unregularized) and gradient flow in two underparameterized regression problems: single-index model \eqref{eqn:single-index-model} and California Housing dataset with the loss function $g(\psi)=\frac{1}{2n}\|\psi-y\|_2^2$ and $n=2048$ randomly-chosen samples. Theorem \ref{thm:gf-convergence} indicates convergence to an \emph{in-class} optimal predictor in $\alpha f(B)$. Furthermore, the output scaling factor $\alpha$ has a self-regularization effect: large $\alpha > 0$ implies a smaller set $B$. We demonstrate the impact of different $\alpha > 0$ and the impact of Gauss-Newton preconditioning in Figure \ref{fig:underpar}. 
\begin{figure}[ht!]
    \centering
    \begin{subfigure}[t]{0.5\textwidth}
        \centering
        \includegraphics[width=.9\linewidth]{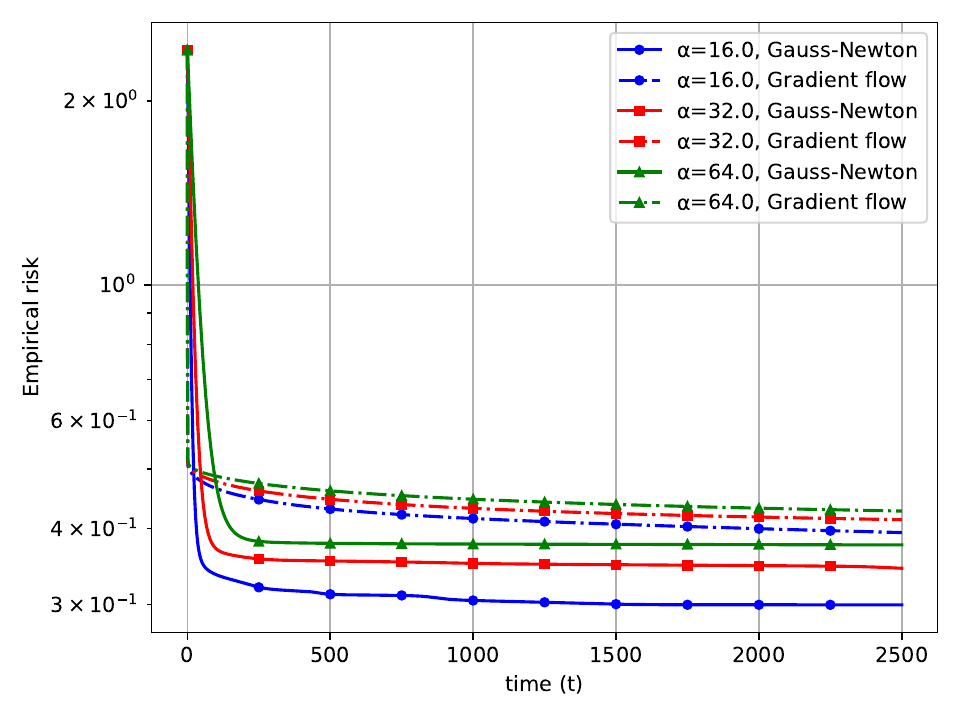}
        \caption{Single-index model}
        \label{fig:underpar-si}
    \end{subfigure}%
    ~ 
    \begin{subfigure}[t]{0.5\textwidth}
        \centering
        \includegraphics[width=.9\linewidth]{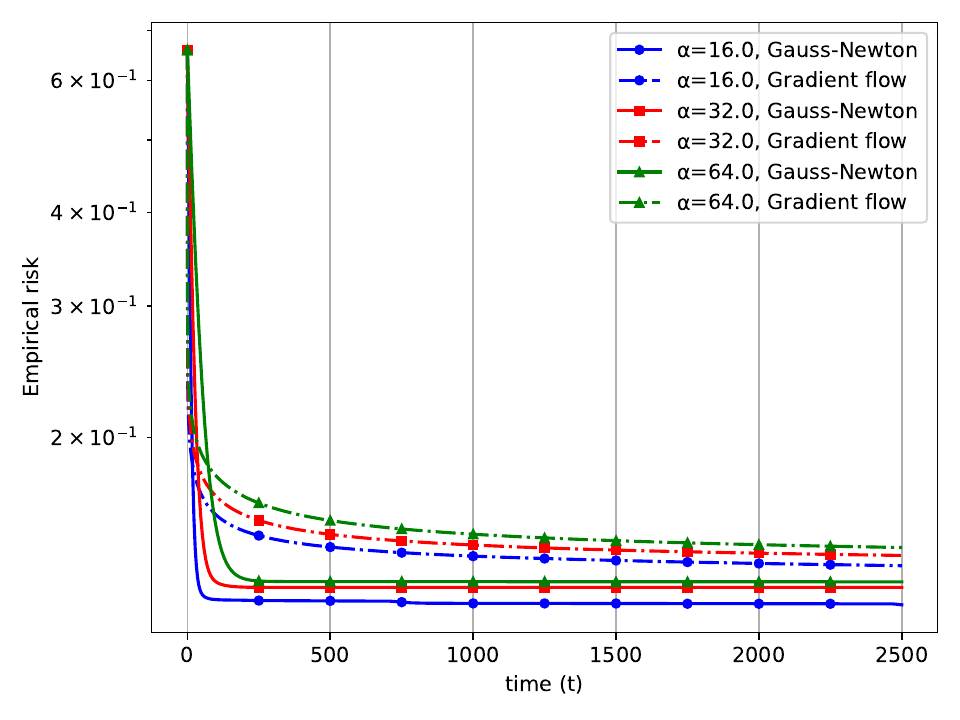}
        \caption{California Housing}
        \label{fig:underpar-ch}
    \end{subfigure}
    \caption{Empirical loss in the underparameterized regime under the Gauss-Newton and gradient flow dynamics for various $\alpha$.}
    \label{fig:underpar}
\end{figure}
A large scaling factor $\alpha$ yields smaller parameter set $B$, thus the in-class optimum predictor has a larger inductive bias as demonstrated in Figure \ref{fig:underpar}, which verifies the regularization impact of $\alpha > 0$ in the underparameterized regime.

\vskip 0.2in
\bibliography{sample}

\end{document}